\documentclass[reqno,12pt,letterpaper]{amsart}
\usepackage{amsmath,amssymb,amsthm,graphicx,mathrsfs,url}
\usepackage{graphicx} 
\usepackage{subfig}
\usepackage[usenames,dvipsnames]{color}
\usepackage[colorlinks=true,linkcolor=Red,citecolor=Green]{hyperref}
\usepackage{amsxtra}

\def\arXiv#1{\href{http://arxiv.org/abs/#1}{arXiv:#1}}

\def\?[#1]{\textbf{[#1]}\marginpar{\Large{\textbf{??}}}}
\def\smallsection#1{\smallskip\noindent\textbf{#1}.}
\let\epsilon=\varepsilon % sorry Knuth

\newcommand{\RR}{{\mathbb R}}

\newcommand{\ZZ}{{\mathbb Z}}

\newtheorem{theo}{Theorem}
\newtheorem{prop}{Proposition}[section]

\newtheorem{ass}{Assumption}
\newtheorem{lemm}[prop]{Lemma}

\newtheorem{rem}{Remark}
\newtheorem{ex}{Example}

\numberwithin{equation}{section}

\let\Im=\Imag

\let\Re=\Real

\DeclareMathOperator{\vol}{vol}

\def\indic{\operatorname{1\hskip-2.75pt\relax l}}
\usepackage{scalerel}

\newcommand\reallywidehat[1]{\arraycolsep=0pt\relax%
\begin{array}{c}
\stretchto{
  \scaleto{
    \scalerel*[\widthof{\ensuremath{#1}}]{\kern-.5pt\bigwedge\kern-.5pt}
    {\rule[-\textheight/2]{1ex}{\textheight}} %WIDTH-LIMITED BIG WEDGE
  }{\textheight} % 
}{0.5ex}\\           % THIS SQUEEZES THE WEDGE TO 0.5ex HEIGHT
#1\\                 % THIS STACKS THE WEDGE ATOP THE ARGUMENT
\rule{-1ex}{0ex}
\end{array}
}

\title[Magnetic Schr\"odinger- and Hartree equations]{Computability of magnetic Schr\"odinger and Hartree equations on unbounded domains}

\author{Simon Becker}
\email{simon.becker@damtp.cam.ac.uk}
\address{Department of Applied Mathematics and Theoretical Physics, University of Cambridge, Wilberforce Road, Cambridge, CB3 0WA, United Kingdom.}

\author{Jonathan Sewell}
\email{js2354@cam.ac.uk}
\address{Department of Applied Mathematics and Theoretical Physics, University of Cambridge, Wilberforce Road, Cambridge, CB3 0WA, United Kingdom.}

\author{Euan Tebbutt}
\email{ejt69@cam.ac.uk}
\address{Department of Applied Mathematics and Theoretical Physics, University of Cambridge, Wilberforce Road, Cambridge, CB3 0WA, United Kingdom.}

\begin{document}
\begin{abstract}
We study the computability of global solutions to linear Schr\"odinger equations with magnetic fields and the Hartree equation on $\mathbb R^3$. We show that the solution can always be globally computed with error control on the entire space if there exist a priori decay estimates in generalized Sobolev norms on the initial state. Using weighted Sobolev norm estimates, we show that the solution can be computed with uniform computational runtime with respect to initial states and potentials. We finally study applications in optimal control theory and provide numerical examples.
\end{abstract}
\maketitle

\section{Introduction}

In this article, we address for the first time the question under which conditions the solution to linear Schr\"odinger equations with external magnetic field or the Hartree equation can be reduced to an effective dynamics on a bounded domain that can then be numerically computed. We show that in both cases, this is possible if the initial state can be controlled in a generalized Sobolev-type norm that controls both the regularity and the decay of the initial state. The linear Schrödinger equation with external magnetic field describes a single-particle system subject to an external magnetic field. The Hartree equation naturally occurs as a mean-field limit of many-particle systems and in quantum transport theory \cite{IZL,BM}. Due to their relevance in physics, we focus on magnetic Schr\"odinger equations or Hartree equations with bilinear interaction potential on $\mathbb R^3$. 

Most of the literature dealing with numerical approximation schemes on unbounded domains, has been concerned with the existence of transparent or absorbing boundary conditions, see e.g. \cite{KG18,JG07} and references therein, for electro-magnetic time-dependent Schrödinger equations. In \cite{KG18} spatial restrictions are imposed on the electromagnetic vector potential, to avoid undesired boundary effects. In this work, we circumvent such undesired restrictions by studying a somewhat different question: \emph{Can we forget about the unboundedness of the domain and just restrict the dynamics to a bounded domain whose size is uniform in the input parameters (initial state, potentials, control) of our problem?} 

Similar to the above results for linear Schrödinger equations, a mode decomposition for nonlinear Schrödinger equations has been thoroughly addressed on bounded domain by Soffer and Stucchio in \cite{Soffer,Soffer2}.

The linear Schrödinger equation is often numerically discretized by a simple Crank-Nicholson method, as this one preserves the $L^2$-norm. To numerically study the non-linear dynamics, there exist many results on splitting methods for nonlinear Hartree equations (also called Schrödinger-Poisson equation)\cite{F,L} and even relativistic \cite{BD} and fractional Hartree equations \cite{ZWW}. 

Thus, even though there exists a vast literature on convergent discretization schemes, an implementation on a physical computer must not rely on infinitely many variables. For instance, a convergent algorithm that assumes a uniform discretization of the entire space $\mathbb R^3$ cannot be implemented in practice. 

In Section \ref{sec:OCT}, we also discuss applications of our work to optimal control problems. 

We illustrate our results with some basic numerical examples in Section \ref{sec:numex}. 

We emphasize at this point that all the constants appearing in this work could be made fully explicit. But for the sake of presentation, we do not specify them explicitly. 

Let us now describe the precise set-up of the equations we consider, starting with the magnetic linear Schrödinger equation:

\subsection{The magnetic Schr\"odinger equation} 

Let $H_0$ be a self-adjoint magnetic Schr\"odinger operator $H_0=(-i\nabla-A)^2+V: D(H_0) \subset L^2(\mathbb R^3) \rightarrow L^2(\mathbb R^3)$ where $V$ is a time-independent \emph{pinning potential} $V$. We consider the evolution of a particle under the influence of an external \emph{control potential} $V_{\operatorname{con}}$ with control function $u \in W^{1,1}_{\operatorname{pcw}}(0,T)$. Writing $V_{\operatorname{TD}}(t):= u(t) V_{\operatorname{con}} $ for the full time-dependent potential, we study time-dependent linear Schr\"odinger equations of the form
\begin{equation}
\begin{split}
\label{eq:bilSchr}
i \partial_t \psi(x,t) &= (H_0 + V_{\operatorname{TD}}(t))\psi(x,t), \quad (x,t) \in \mathbb R^3 \times (0,T) \\
\psi(\bullet,0)&=\varphi_0.
\end{split}
\end{equation}

The assumptions we impose on the potentials are as follows
\begin{ass}
\label{ass:ass1}
We demand that the static pinning potential can be written as a sum of a regular (possibly unbounded at infinity) and singular part (possible local singularities) $V=W_{\operatorname{reg}}+W_{\operatorname{sing}}$. For these two potentials and the control potential $V_{\operatorname{con}}$ with control $u$, we impose the assumptions that
\begin{itemize}
    \item $\Vert \langle \bullet \rangle^{-2} V_{\operatorname{con}} \Vert_{\infty}< \infty$ and $\Vert \langle \bullet \rangle^{-2} W_{\operatorname{reg}} \Vert_{\infty}< \infty,$
    \item $W_{\operatorname{sing}} \in L^2(\RR^3),$ and
    \item $u \in W^{1,1}_{\operatorname{pcw}}(0,T).$
\end{itemize}
For the magnetic vector potential we either assume a constant magnetic field $B_0>0$ with associated vector potential $A=\frac{B_0}{2}(-x_2,x_1,0)$ or a vector potential $A \in W^{1,\infty}(\RR^3;\RR^3)$.
\end{ass}

\begin{rem}
Our choice of singular and regular part allows us to treat standard physical examples of potentials such as the Coulomb potential $x \mapsto 1/\vert x \vert$ (singular part) and the harmonic potential $x \mapsto \vert x \vert^2$ (regular part and control potential).
\end{rem}

The Schr\"odinger equation \eqref{eq:bilSchr} with linear control potential appears naturally in the study of static physical systems with Hamiltonian $H_0$, under the influence of a time-dependent bilinear electric potential.

In this article, we build upon techniques, introduced in \cite{B,BKP}, to prove existence of solutions to Schrödinger equations in certain weighted Sobolev spaces that ensure additional spatial decay. These spaces are essential in our study of global numerical algorithms that provide solutions to \eqref{eq:bilSchr} on unbounded domains. 

\subsection{Hartree equation} 
In the case of the Hartree equation, we consider a single particle described by the Hartree equation with Schr\"odinger operator $H_0=(-i\nabla-A)^2+V: D(H_0) \subset L^2(\mathbb R^3) \rightarrow L^2(\mathbb R^3)$ and static \emph{pinning potential} $V$ and \emph{control potential} $V_{\operatorname{con}}$ with time-dependent control function $u \in W^{1,1}(0,T)$ such that $V_{\operatorname{TD}}(t):= u(t) V_{\operatorname{con}} $
\begin{equation*}
\begin{split}
%\label{eq:HF}
i \partial_t \psi(x,t) &= (H_0 + V_{\operatorname{TD}}(t))\psi(x,t)+\left(\vert \psi(\bullet,t) \vert^2*\frac{1}{\vert \bullet \vert} \right)(x)\psi(x,t), \quad (x,t) \in \mathbb R^3 \times (0,T) \\
\psi(\bullet,0)&=\varphi_0.
\end{split}
\end{equation*}
The potentials and magnetic fields are assumed to satisfy the conditions in Assumption \ref{ass:ass1}.

In this section, we shall fix some notation that we use throughout the article: 

\smallsection{Notation} We denote by $H^n$ the standard Sobolev space with respect to the $L^2$ inner product. By $H^k_n$ we denote the weighted Sobolev spaces with scalar product induced by 
$$\Vert f \Vert_{H^k_n} := \sqrt{\Vert f \Vert_{H^k}^2 + \Vert \vert x \vert^{n} f \Vert^2_{2}}.$$
We write $a \lesssim b$ to indicate that there is a constant $C>0$ such that $a \le C b.$ The space $W^{1,1}_{\operatorname{pcw}}(0,T)$ is the space of piecewise $W^{1,1}$ functions on $(0,T).$ We denote by $C^k_b$ the Banach space of $C^k$ functions whose first $k$ derivatives and the function itself are all globally bounded. We denote by $\mathcal L(X)$ the bounded linear operator on some normed space $X.$

We denote by $C>0$ an arbitrary constant and $C_{\mu}>0$ an arbitrary constant that depends on some parameter $\mu.$

We sometimes abuse the notation for spatial integrals in $\mathbb R^3$ to simplify the notation and mean the same when writing the following expressions, if there is no misunderstanding possible,
$$\int_{\mathbb R^3} f(x) \ dx = \int_{\mathbb R^3} f \ dx = \int_{\mathbb R^3} f.$$
In particular, we also abuse the notation by identifying the following expressions and versions of that
\[ \int_{\mathbb R^3} (1+\vert x \vert^2)f = \int_{\mathbb R^3} (1+\vert x \vert^2)f(x) =\int_{\mathbb R^3} (1+\vert x \vert^2)f(x) \ dx.\]

We recall the following elementary bound which links control on weighted Sobolev norms to spatial decay in the $L^2$-sense
\begin{rem}
Let $\phi$ be a function in $H_{\eta}(\RR^3)$ then it follows that for all $R>0$
\label{lm:InequalityForNormOutsideBall}
\[        ||\phi\indic_{B_R(0)^c}||_{L^2(\RR^3)}\leq||\phi||_{H_\eta(\RR^3)}(1+R^2)^{-\eta/2}\leq||\phi||_{H_\eta(\RR^3)}R^{-\eta}.\]
\end{rem}
\smallsection{Outline of the article}
\begin{itemize}
\item In Section \ref{sec:magn}, we derive estimates on weighted solutions to the magnetic linear Schrödinger equation. Theorem \ref{theo:H2ExistBoundedMag} contains bounds on solutions to the linear Schrödinger equation with bounded magnetic fields and Theorem \ref{theo:magfield} contains such estimates for constant magnetic fields. In Theorem \ref{theo:Dirichlet}, we prove a quantitative reduction to an auxiliary boundary value problem.
\item In Section \ref{sec:Hartree}, we derive weighted bounds on solutions to the Hartree equation, see Theorem \ref{theo:HF} and prove a quantitative reduction to a boundary value problem in Theorem \ref{theo:auxBVP}.
\item In Section \ref{sec:numerics}, we introduce a convergent numerical discretization scheme and prove the convergence with an explicit rate, cf. Theorem \ref{theo:numerics}.
\item Section \ref{sec:OCT} discusses applications of our results to optimal control theory.
\item The final Section \ref{sec:numex} illustrates numerical consequences of our analysis.
\end{itemize}
\section{The magnetic linear Schrödinger equation}
\label{sec:magn}
We start by recalling the following classical Lemma \cite{BH,BKP} that shows that the free Schrödinger equation is well-posed in the weighted Sobolev space $H^2_2$:
\begin{lemm}
\label{lemm:freegroup}
Let $\varphi_0 \in H^2_2$.
There exists a $\psi \in C([0,T];H^2_2)$ solving 
\begin{equation*}\begin{split}
i \partial _t \psi &= -\Delta \psi \text{ with initial condition }\psi(x,0) = \varphi_0(x)        
\end{split}\end{equation*}
such that $\left\lVert \psi(\bullet, t) \right\rVert_{H^2_2} \leq C_T \left\lVert \varphi_0(\bullet) \right\rVert_{H^2_2}.$
\end{lemm}
As a first step to study the magnetic linear Schrödinger equation, we show the existence of solutions to a regularized version of the magnetic linear Schrödinger equation with discrete derivatives, for some regularization parameter $h>0,$
\begin{equation}
\label{eq:discderiv}
(\nabla_{h}\psi)(x):=\left(\frac{\psi(x+he_1)-\psi(x)}{ h},..,\frac{\psi(x+he_3)-\psi(x)}{ h}\right)
\end{equation}
coupled to a bounded magnetic vector potential and regularized potentials. 
\begin{lemm}
\label{lemm:existenxe}
Let $V \in L^\infty((0,T);C_b^2(\RR ^3))$ and $A \in W^{2,\infty}(\RR ^3;\RR^3)$. 
For initial states $\varphi_0 \in H^2_2(\RR^3)$, the equation
\begin{equation}
\begin{split}
\label{eq:schroedinger}
    i \partial _t \psi(x,t) &=- \Delta \psi(x,t) + V(x,t) \psi(x,t) + iA(x) \cdot (\nabla_{h}  \psi)(x,t)+i (\nabla_{-h} \cdot (A\psi))(x,t)  \\
    \psi(x,0) &= \varphi_0(x)
\end{split}
\end{equation}
has a unique mild solution $ \psi \in C([0,T];H^2_2(\RR^3))$. 
Furthermore, if $\rho = \left \lVert V \right \rVert _{L^\infty(0,T;C_b^2(\RR ^3))}$, $\mu = \left \lVert A \right \rVert _{W^{2,\infty}}$, then there exists $C_{T,\rho,\mu,h}$ such that for all $t \in [0,T]$, 
$\left\lVert \psi(\bullet, t) \right\rVert_{H^2_2} \leq C_{T,\rho,\mu,h} \left\lVert \varphi_0(\bullet) \right\rVert_{H^2_2}. $
\end{lemm}
\begin{proof}
Let $Y = C([0,T];H^2_2(\RR^3))$ with norm $\lVert \psi \rVert_Y = \sup_{t \in [0,T]} e^{-\lambda t}  \left\lVert \psi(\bullet, t) \right\rVert_{H^2_2(\RR^3)}$ for some $\lambda > 0$ to be specified later. In terms of the one-parameter group $S(t)=e^{i \Delta t}$, solving the equation \eqref{eq:schroedinger} is equivalent to solving the fixed-point equation
\[\psi(t) = S(t)\varphi_0 - i \int_0^t S(t-s) \left(\psi(s)V(s) + (iA \cdot \nabla_{h} \psi)(s) + (i\nabla_{-h} \cdot (A\psi))(s) \right) ds.   \]

%since this implies
%\begin{equation*}
%    \begin{split}
 %       \psi'(t)&= i \Delta S(t) \varphi_0 +i \left(\psi(t)V(t) + iA \cdot \nabla_{\chi}\psi(t) \right) -\Delta S(t) \int_0^t S(-s) \left(\psi(s)V(s) + iA \cdot \nabla_{\chi}\psi(s) \right) \ ds\\
 %       &=i \Delta \left(S(t) \varphi_0 + i \int_0^t S(t-s) \left(\psi(s)V(s) + iA \cdot \nabla_{\chi}\psi(s) \right) \ ds\right) + i \left(\psi(t)V(t) + iA \cdot \nabla_{\chi}\psi(t) \right) \\
%        &= i \Delta \psi(t) +i V(t) \psi(t) -A \cdot \nabla_{\chi}\psi(t).
 %   \end{split}
%\end{equation*}
%as required. \\

Let $B_R = \{ \psi \in Y : \lVert \psi \rVert_Y \leq R \}$, for some $R$ to be specified later, and define a map $\Phi: B_R \to B_R$ by 
\[
    \Phi(\psi(t)) = S(t)\varphi_0 - i \int_0^t S(t-s)\left(\psi(s)V(s) + (iA \cdot \nabla_{h} \psi)(s) + (i\nabla_{-h} \cdot (A\psi))(s) \right) ds.
\]

To see that $\Phi$ maps $B_R$ into itself, first notice that
\begin{equation*}
\begin{split}
 \left\lVert \psi(s)V(s) \right\rVert_{H^2_2}  &\lesssim \left\lVert \Delta (\psi(s)V(s)) \right\rVert_{2} + \left\lVert (1+|\bullet|^2) \psi(\bullet,s)V(\bullet,s) \right\rVert_{2} \\
 &\lesssim \left\lVert \Delta \psi(s)V(s) + \nabla \psi(s) \cdot \nabla V (s) + \psi \Delta V \right\rVert_{2} + \rho \lVert \psi(s) \rVert_{H_2} \le C_\rho \left\lVert \psi(s) \right\rVert_{H^2_2}.
\end{split}
\end{equation*}

We then use in addition that $[\nabla_h, \Delta]=0$
\begin{equation*}
\begin{split}
  \left\lVert A \cdot \nabla_{h}\psi(s) \right\rVert_{H^2_2}  &\lesssim \left\lVert \Delta (A \cdot \nabla_{h}\psi(s)) \right\rVert_{2} + \left\lVert (1+|\bullet|^2) A(\bullet) \cdot \nabla_{h}\psi(\bullet,s) \right\rVert_{2} \\
 &\leq C_{\mu, h} \left\lVert \psi(s) \right\rVert_{H^2_2}
\end{split}
\end{equation*}
and
\begin{equation*}
\begin{split}
  \left\lVert \nabla_{-h} \cdot (A \cdot \psi)(s) \right\rVert_{H^2_2}  &\lesssim \left\lVert \Delta (\nabla_{-h} \cdot (A \cdot \psi)(s)) \right\rVert_{2} + \left\lVert (1+|\bullet|^2) \nabla_{h} \cdot (A \cdot \psi)(s) \right\rVert_{2} \\
 &\leq C_{h}  (\left\lVert \Delta (A \cdot \psi)(s) \right\rVert_{2} +  \mu \lVert \psi(s) \rVert_{H_2}) \leq C_{\mu, h} \left\lVert \psi(s) \right\rVert_{H^2_2}.
\end{split}
\end{equation*}
Let $C_T$ be the constant from Lemma \ref{lemm:freegroup}, let $2(C_\rho + C_{\mu, h}) C_T \leq \lambda $ , and let $R \ge 2 C_T \left\lVert \varphi_0 \right\rVert_{H^2_2}.$
Hence, we find that, $\Phi$ is indeed a contraction:

\begin{equation*}
    \begin{split}
    \lVert \Phi(\psi_1) - \Phi(\psi_2) \rVert_Y & = \sup_t \left( e^{-\lambda t}\left\lVert  \int_0^t S(t-s)  \left(V(s) + iA \cdot \nabla_{h}+ i\nabla_{-h} \cdot A  \right)   (\psi_1 - \psi_2)(s) ds \right\rVert_{H^2_2} \right) \\
    & \leq C_T (C_\rho + C_{\mu, h}) \sup_t \left( e^{-\lambda t} \int_0^t e^{\lambda s}e^{-\lambda s} \left\lVert (\psi_1 - \psi_2)(s) \right\rVert_{H^2_2} ds \right) \\
    & \leq C_T (C_\rho + C_{\mu, h}) \lVert \psi_1 - \psi_2 \rVert_Y \sup_t \left( e^{-\lambda t} \int_0^t e^{\lambda s} ds \right) \leq \frac{1}{2} \lVert \psi_1 - \psi_2 \rVert_Y.
\end{split}
\end{equation*}
Therefore $\Phi$ has a unique fixed point in $Y$.

The previous estimate immediately implies that $\left\lVert \psi(\bullet, t) \right\rVert_{H^2_2} \leq C_{T,\rho, \mu, h} \left\lVert \varphi_0(\bullet) \right\rVert_{H^2_2} $.
\end{proof}
\begin{rem}
We want to mention at this point that both Lemma \ref{lemm:freegroup} and Lemma \ref{lemm:existenxe} hold in $H^1_1$ as well by very similar arguments. 
\end{rem}
In the following Lemma, we get rid of the discrete differentiation of the magnetic vector potential that we used in Lemma \ref{lemm:existenxe} to show the existence of solutions.
\begin{lemm}
\label{lemm:H11}
Let $T>0$ and potentials as in Assumption \ref{ass:ass1} such that $ \left\lVert u \right \rVert _{W^{1,1}(0,T)}  = \alpha$, $\left\lVert |\bullet|^{-2}V_{\operatorname{con}}(x) \right \rVert _{L^\infty}  \leq \rho,$ and $\left\lVert |\bullet|^{-2}W_{\operatorname{reg}}(x) \right \rVert _{L^\infty}  \leq \rho.$     
For vector potentials $A  \in W^{1,\infty}(\RR^3;\RR^3)$ with $\lVert A \rVert_{W^{1,\infty}} = \mathcal{A}$, 
there exists a constant $C_{T,\alpha,\rho}$ such that for any $\varphi_0 \in H^1_1$, the equation
\begin{equation*}\begin{split}
    i \partial _t \psi(x,t) &= ( -i\nabla -A)^2 \psi(x,t) + V \psi(x,t)  \\
    \psi(x,0) &= \varphi_0(x)
\end{split}\end{equation*}
has a unique solution $\psi \in L^\infty(0,T;H^1_1)$, and $\psi$ satisfies $   \left \lVert \psi \right \rVert _{L^\infty(0,T;H^1_1)} \leq C_{T,\alpha,\rho,\mathcal A} \lVert \varphi_0 \rVert _{H^1_1}.$
\end{lemm} 

\begin{proof}
First, approximate the potentials by elements of $L^\infty(0,T;C_b^2(\RR ^3))$ in the following way:

Let $\chi \in C^\infty_c (\RR^3)$,  where $0 \leq \chi$, chosen symmetric $\chi(x) = \chi(-x)$ and normalized $\int_{\RR^3} \chi (x) dx = 1$. We define the rescaled function $\zeta_\epsilon (x) = \frac{1}{\epsilon^3} \chi \left( \frac{x}{\epsilon} \right).$ Using the truncation map $(\operatorname{Trunc}_\epsilon f)(x) = \operatorname{sign}(f(x)) \max \left( \frac{1}{\epsilon} , f(x) \right)$, we define for $i \in \{\operatorname{reg}, \operatorname{con}\}$
\begin{equation}
\label{eq:regularizedV}
 V_i^\epsilon (x)  :=( \zeta_\epsilon * \operatorname{Trunc}_\epsilon (V_i))(x) = \int_{\RR^3} \operatorname{Trunc}_\epsilon (V_i(x+\epsilon y)) \chi (y) dy.
 \end{equation}
 Finally, let $W_{\operatorname{sing}}^{\varepsilon}:=\zeta_\epsilon*W_{\operatorname{sing}}.$

 By properties of convolutions it can be shown that 
 \begin{equation*}
 \begin{split}
   \left\lVert (1+|\bullet|^{2})^{-1}W_{\operatorname{reg}}^\epsilon(\bullet) \right \rVert _{L^\infty}  &\leq  \rho\text{ and }  
   \left\lVert(1+|\bullet|^{2})^{-1} V_{\operatorname{con}}^\epsilon(\bullet) \right \rVert _{L^\infty}  \leq  \rho.
 \end{split}
 \end{equation*}
 
For the discrete derivative 
\begin{equation}
\label{eq:discd}
 D_{h} \psi =-i\nabla_h \psi,
\end{equation} with $\nabla_h$ as in \eqref{eq:discderiv}, we have by Lemma \ref{lemm:existenxe}, that for all $\epsilon>0$ , there exists $\psi_\epsilon \in L^\infty(0,T ; H^2_2 )$ satisfying
 \begin{equation*}\begin{split}
    i \partial _t \psi(x,t)_\epsilon &=(H_B^\epsilon + V^\epsilon) \psi_\epsilon(x,t)  \\
    \psi_\epsilon(x,0) &= \varphi_0(x)
\end{split}\end{equation*}
where $H_B^\epsilon = -\Delta  - A D_{\epsilon}-D_{-\epsilon} A + A \cdot A$ is the regularized magnetic Schrödinger operator.
We then intend to bound the $H^1_1$ norm of $\psi_\epsilon$ uniformly in $\epsilon$.

We multiply the preceding equation by $(1+\vert x \vert^2)\bar \psi_{\varepsilon}(x)$ and integrate over $\RR^3$

\[\int_{\mathbb R^3} (1+\vert x \vert^2) (\overline{\psi_{\varepsilon}}i\partial_t  \psi_{\varepsilon})(x,t) + (1+\vert x \vert^2) (\overline{\psi_{\varepsilon}} H_B^\epsilon \psi_{\varepsilon})(x,t) +  (1+\vert x \vert^2)V^{\varepsilon}(x) \vert \psi_{\varepsilon}(x,t) \vert^2 \ dx =0.\]

Thus by taking the imaginary part, we find 

\[ \Im \int_{\mathbb R^3} (1+\vert x \vert^2) (\overline{\psi_{\varepsilon}}i\partial_t  \psi_{\varepsilon})(x,t) - (1+\vert x \vert^2)( \psi_{\varepsilon} H_B^\epsilon \psi_{\varepsilon})(x,t)  \ dx =0.\]

In the first term, we note that $\Im(i\overline{\psi_{\varepsilon}}\partial_t  \psi_{\varepsilon} )= \Re(\overline{\psi_{\varepsilon}}\partial_t  \psi_{\varepsilon})= \frac{1}{2} \partial_t \vert \psi_{\varepsilon}\vert^2.$ This implies that 
\begin{equation*}
\begin{split}
    \partial_t \int_{\mathbb R^3} (1+\vert x \vert^2) \vert \psi_{\varepsilon}(x,t) \vert^2 \ dx & =2\Im \int_{\RR^3} (1+\vert x \vert^2) (\overline{\psi_{\varepsilon}} H_B^\epsilon \psi_{\varepsilon} )(x,t) \ dx \\
     &=2\Im \int_{\RR^3} (1+\vert x \vert^2) (\overline{\psi_{\varepsilon}} \left( -\Delta - A D_{\epsilon}-D_{-\epsilon} A \right) \psi_\epsilon)(x,t) \ dx.
\end{split}
\end{equation*}  

Integration by parts yields 
\begin{equation*}
    \begin{split}
-2 \Im \int_{\RR^3} (1+\vert x \vert^2) (\bar \psi_{\varepsilon} \Delta \psi_{\varepsilon})(x)  \ dx 
&=2\Im \int_{\RR^3}\underbrace{(1+\vert x \vert^2) \vert \nabla \psi_{\varepsilon}(x) \vert^2}_{\in \mathbb R} \ dx +2 \Im \int_{\RR^3} 2 \bar \psi_{\varepsilon}(x) x \cdot \nabla   \psi_{\varepsilon}(x) \ dx \\
 &= \Im \int_{\RR^3} 4 \overline{ \psi_{\varepsilon}(x)} x \cdot \nabla   \psi_{\varepsilon}(x) \ dx  \leq  \int_{\RR^3} 4  |\psi_{\varepsilon}(x)| |x|  |\nabla \psi_{\varepsilon}(x)| \ dx.
    \end{split}
\end{equation*}
Hence, using that $ab \le \frac{a^2}{2}+\frac{b^2}{2}$, this implies
\begin{equation*}
    \begin{split}
      -2 \Im \int_{\RR^3} (1+\vert x \vert^2) \bar \psi_{\varepsilon} \Delta \psi_{\varepsilon}  \ dx  &\leq 2 \left( \int_{\RR^3} |x|^2 \vert \psi_{\varepsilon}\vert^2 \ dx+  \int_{\RR^3} \vert \nabla   \psi_{\varepsilon}\vert^2  dx \right).  
    \end{split}
\end{equation*}

In terms of $F(x)=(1+\vert x \vert^2)A(x)$, we find

\begin{equation*}
    \begin{split}
\int_{\RR^3}(1+\vert x\vert^2) \overline{\psi(x)}((A D_{\epsilon} + D_{-\epsilon} A) \psi)(x) \ dx &=\int_{\RR^3} \overline{\psi(x)} ((F D_{\epsilon} + D_{-\epsilon} F) \psi)(x) \ dx \\
& \quad + \int_{\RR^3}\overline{\psi(x)}([(1+\vert x \vert^2),D_{-\epsilon}]A \psi)(x) \ dx.
\end{split}
\end{equation*}

We now observe that $\int_{\RR^3}\overline{\psi(x)} ((F D_{\epsilon} + D_{-\epsilon} F) \psi)(x) \ dx \in \RR$ which is because $F D_{\epsilon} +D_{-\epsilon} F$ is self-adjoint.

Observe then that for $g(x)=(f(x-he_1)(D_{-\varepsilon})_1x^2,..,f(x-he_1)(D_{-\varepsilon})_3x^2)$, where $(D_{-\varepsilon})_i$ is the discretized derivative, \eqref{eq:discd}, in direction $i$, we have
\begin{equation*}
    \begin{split}
([(1+x^2),D_{-\epsilon}]f)(x) 
&=(1+x^2)(D_{-\epsilon} f)(x) -D_{-\epsilon}((1+x^2)f(x))\\
&= (1+x^2) D_{-\epsilon}f(x) - (1+x^2) D_{-\epsilon}f(x) +g(x)= g(x).
    \end{split}
\end{equation*}
This implies that
\begin{equation*}
    \Im \int_{\RR^3}(1+\vert x\vert^2) \overline{\psi(x)}((AD_{\epsilon}+D_{-\epsilon}A) \psi)(x) \ dx= 2\int_{\RR^3} \overline{\psi(x)}\sum_{i=1}^3 A_i\psi(x-he_i)(D_{-\varepsilon})_ix^2 \ dx.
\end{equation*}

This last term can then be estimated by $\int_{\RR^3} (1+x^2) \vert \psi(x) \vert^2$, since we assume in this Lemma that $A$ is bounded. Therefore, 
\begin{equation}
\label{eq:timeestm}
\frac{d}{dt} \lVert \psi_\epsilon \rVert _{H_1}^2 \leq (2+2\mathcal{A})\lVert \psi_\epsilon \rVert _{H_1}^2 + 2\lVert \psi_\epsilon \rVert _{H^1}^2.
\end{equation}
Starting from  $i\partial_t \psi_{\varepsilon} = H_B^\epsilon \psi_{\varepsilon} + V^{\varepsilon} \psi_{\varepsilon},$ we multiply this equation by $\partial _t \bar \psi_{\varepsilon}$, take the real part, and integrate the expression over $\mathbb R^3$. This yields
\begin{equation}
\begin{split}
\label{eq:identity1234}
    0 &= \Re \int_{\RR^3}- \Delta \psi_{\varepsilon} \partial _t \bar \psi_{\varepsilon}  - (\partial _t \bar \psi_{\varepsilon}) (AD_{\epsilon} +D_{-\epsilon}A)\psi_\epsilon + A \cdot A \psi_\epsilon \partial _t \bar \psi_{\varepsilon}  + V^{\varepsilon} \psi_{\varepsilon} \partial _t \bar \psi_{\varepsilon} \ dx.
\end{split}
\end{equation}
To treat the second term, in the above integrand, notice that
\begin{equation*}
\begin{split}
      \int_{\RR^3} \partial _t (\bar \psi_{\varepsilon} (AD_{\epsilon} +D_{-\epsilon}A)\psi_\epsilon)(x) \ dx 
      =2 \Re \int_{\RR^3} (\partial _t \bar \psi_{\varepsilon})(x) (AD_{\epsilon} +D_{-\epsilon}A)(x)\psi_\epsilon(x) \  dx
\end{split}
\end{equation*}
by self-adjointness.
Therefore, the identify in \eqref{eq:identity1234} reduces to 
\begin{equation*}
\begin{split}
    0 & = \Re \int_{\RR^3}  \nabla \psi_{\varepsilon} \cdot  \partial _t \nabla \bar \psi_{\varepsilon}   - (\partial _t \bar \psi_{\varepsilon}) (AD_{\epsilon} +D_{-\epsilon}A)\psi_\epsilon + |A|^2 \psi_\epsilon \partial _t \bar \psi_{\varepsilon} +  V^{\varepsilon}\psi_{\varepsilon} \partial _t \bar \psi_{\varepsilon} \ dx  \\
    & =  \frac{1}{2}  \int_{\RR^3} \partial _t |\nabla \psi_{\varepsilon}|^2  - \partial _t (\bar \psi_{\varepsilon} (AD_{\epsilon} +D_{-\epsilon}A)\psi_\epsilon) + \partial_t |A \psi_\epsilon|^2  + V^{\varepsilon} \partial _t \left( |\psi_{\varepsilon}|^2 \right) \ dx
\end{split}
\end{equation*}
and hence

\begin{equation}
\begin{split}
\label{eq:auxestm}
    \frac{d}{dt} \int_{\RR^3}  |\nabla \psi_{\varepsilon}|^2 \ dx   & =  \int_{\RR^3}   (\partial _t V^{\varepsilon}) |\psi_{\varepsilon}|^2 \ dx  \\
    &\quad + \frac{d}{dt} \int_{\RR^3} (\bar \psi_{\varepsilon} (AD_{\epsilon} +D_{-\epsilon}A)\psi_\epsilon) - |A \psi_\epsilon|^2 - V^{\varepsilon} |\psi_{\varepsilon}|^2 \ dx.
\end{split}
\end{equation}

In addition, we recall that
\begin{equation}
\begin{split}
\label{eq:auxestm2}
     \int_{\RR^3} \vert \partial _t V^{\varepsilon} \vert |\psi_{\varepsilon}|^2 \  & \leq |u'(t)| \left \lVert \frac{V_{\operatorname{con}}(\bullet)} {1+|\bullet|^2} \right \rVert _{L^\infty} \lVert \psi_\epsilon (\bullet, t) \rVert_{H^1}^2.
\end{split}
\end{equation}

We now consider an energy $E_\lambda^\epsilon(t)$ defined for $\lambda>0$ to be chosen later
\[
E_\lambda^\epsilon(t) = \lVert \psi_\epsilon(t) \rVert_{H^1}^2 + \lambda \lVert \psi_\epsilon(t) \rVert_{H_1}^2.
\]
Using the preceding estimates and identities \eqref{eq:timeestm}, \eqref{eq:auxestm}, and  \eqref{eq:auxestm2}, we see that
\begin{equation}
\begin{split}
\label{eq:identity}
    \frac{d}{dt}E_\lambda^\epsilon(t) &= \frac{d}{dt} \int_{\RR^3}  |\nabla \psi_{\varepsilon}|^2 \ dx + \lambda \frac{d}{dt} \int_{\RR^3}  (1+|x|^2) |\psi_{\varepsilon}|^2 \ dx \\
     &\leq \frac{d}{dt} \int_{\RR^3}(\bar \psi_{\varepsilon} (AD_{\epsilon} +D_{-\epsilon}A)\psi_\epsilon) - |A \psi_\epsilon|^2 - V^{\varepsilon} |\psi_{\varepsilon}|^2 \ dx + |u'(t)| \left \lVert \frac{V_{\operatorname{con}}^{\varepsilon}(\bullet)} {1+|\bullet|^2} \right \rVert _{L^\infty} \lVert \psi_\epsilon  \rVert_{H^1}^2\\
    & \quad + 2\lambda \int_{\RR^3}  |x|^2 |\psi_{\varepsilon}|^2 \ dx +2\lambda \int_{\RR^3} |\nabla \psi_{\varepsilon}|^2 \ dx \\
    &\leq \frac{d}{dt} \int_{\RR^3}(\overline{ \psi_{\varepsilon}(t)} (AD_{\epsilon} +D_{-\epsilon}A)\psi_\epsilon(t))- |A \psi_\epsilon|^2-V^{\varepsilon} |\psi_{\varepsilon}|^2 \ dx + C \left( 1 + |u'(t)|  \right) E_\lambda^\epsilon(t)
\end{split}
\end{equation}
where $C$ is a constant that depends only on $T$, $\rho,$ and $\lambda$.
Let $\beta(t) = 1 + |u'(t)|$, integration of \eqref{eq:identity} from $0$ to $t$ shows that using the Cauchy-Schwarz inequality and Lemma \ref{EuanRulesForF}
\begin{equation*}
\begin{split}
    E_\lambda^\epsilon(t) &\leq \int_{\RR^3} (\overline{ \psi_{\varepsilon}(t)} (AD_{\epsilon} +D_{-\epsilon}A)\psi_\epsilon(t)) - |A \psi_\epsilon(t)|^2 - V^{\varepsilon}(t) |\psi_{\varepsilon}(t)|^2 \  \\
    & \quad + \left| \int_{\RR^3} (\overline{ \psi_{\varepsilon}(0)} (AD_{\epsilon} +D_{-\epsilon}A)\psi_\epsilon(0)) - |A \psi_\epsilon(0)|^2 - V^{\varepsilon}(0) |\psi_{\varepsilon}(0)|^2 \  \right| \\
    & \quad + C \int _0^t \beta(s) E_\lambda^\epsilon(s) \ ds + E_\lambda^\epsilon(0).
\end{split}
\end{equation*}

Notice that by \eqref{eq:regularizedV}
\begin{equation*}
\begin{split}
    \int_{\RR^3} |W_{\operatorname{sing}}^{\varepsilon}(x)| |\psi_{\varepsilon}(x,t)|^2 \ dx  &  =   \int_{\RR^3} \int_{\RR^3} | \psi_{\varepsilon}(x,t)^2 W_{\operatorname{sing}}(x - \epsilon y) \chi(y) | \ dy \ dx  \\
    & =   \int_{\RR^3} \chi(y) \int_{\RR^3} | \psi_{\varepsilon}(x + \epsilon y,t)^2 W_{\operatorname{sing}}(x)  | \  dx \ dy \\
    & =  \int_{\RR^3} \chi(y) \left \langle |W_{\operatorname{sing}}(\bullet)| \psi_{\varepsilon}(\bullet + \epsilon y,t), \psi_{\varepsilon}(\bullet + \epsilon y,t)  \right \rangle_{L^2}  \ dy.
\end{split}
\end{equation*}

Therefore, by the relative zero-boundedness of the singular potential with respect to the Laplacian, \cite[Remark 3.8]{BH}, it follows that for all $\delta>0$ , there exists $C_\delta$ such that for all $t \in [0,T]$ :
\begin{equation*}
\begin{split}
    \int_{\RR^3} |W_{\operatorname{sing}}^{\varepsilon}(x)| |\psi_{\varepsilon}(x,t)|^2 \ dx & \leq    \int_{\RR^3} \left( \delta \lVert \nabla \psi_{\varepsilon}(\bullet + \epsilon y,t) \rVert_{2}^2 + C_\delta \lVert \psi_{\varepsilon}(\bullet + \epsilon y,t) \rVert_{2}^2 \right) \chi(y) \ dy  \\
    & =    \delta \lVert \nabla \psi_{\varepsilon}(\bullet,t) \rVert_{2}^2 + C_\delta \lVert \psi_{\varepsilon}(\bullet,t) \rVert_{2}^2.
\end{split}
\end{equation*}

Notice also that straight from the assumptions on the magnetic vector potential and potentials
\begin{equation*}
\begin{split}
\int_{\RR^3} |W_{\operatorname{reg}}^{\varepsilon}(x)| |\psi_{\varepsilon}(x,t)|^2 \ dx  \leq \rho \lVert \psi_\epsilon(t)  \rVert_{H_1}^2 
    \text{ and } \quad \int_{\RR^3} |A(x) |\psi_{\varepsilon}(x,t)|^2 \ dx  \leq \lVert A \rVert^2_\infty \lVert \psi_\epsilon(t)  \rVert_{H_1}^2.
\end{split}
\end{equation*}

Then it follows that for all $\kappa>0$ 
\begin{equation*}
    \begin{split}
    \left\vert \int _{\RR^3} \bar \psi_{\varepsilon}(x,t) (A(x)D_{\epsilon} +D_{-\epsilon}A(x))\psi_\epsilon(x,t)) \  \right\rvert & \lesssim \frac{\Vert \psi_{\varepsilon}(t) \Vert^2_{2}}{\kappa}  +  \kappa \mathcal{A} \Vert \psi_{\varepsilon}(t) \Vert^2_{H^1}
    \end{split}
\end{equation*}
since both $\Vert D_{\epsilon}f \Vert_{\infty} \le \Vert \nabla f \Vert_{\infty}$ and $\Vert D_{\epsilon} f \Vert_{2} \le \Vert f\Vert_{H^1}$ for suitable $f.$

 Putting all these together, picking $\delta, \kappa$  sufficiently small, and evaluating $E_\lambda^\epsilon(0)$, we get for all $\lambda$,  
\[E_\lambda^\epsilon(t) \leq \frac{1}{2} \lVert \psi_\epsilon(t)  \rVert_{H^1}^2 + C_{\alpha,\mathcal{A},\rho,T} \lVert \psi_\epsilon  \rVert_{H_1}^2 +  C_{\alpha,\mathcal{A},\rho,\lambda,T} \lVert \varphi_0  \rVert_{H^1_1}^2 +  C \int _0^t \beta(s) E_\lambda^\epsilon(s) \ ds.\]
Pick $\lambda > 2C_{\alpha,\mathcal A,\rho,T}$, and we get 
\[ \lVert \psi_\epsilon(t)  \rVert_{H^1_1}^2 \lesssim \lVert \varphi_0  \rVert_{H^1_1}^2 +  \int _0^t \beta(s) \lVert \psi_\epsilon(t)  \rVert_{H^1_1}^2 \ ds.\]
Then, by Gronwall's inequality, we get  $\lVert \psi_\epsilon(t)  \rVert_{H^1_1} \lesssim  \lVert \varphi_0  \rVert_{H^1_1}.$
\end{proof}

As our first theorem shows, the solution even exists in $H^2_2$.

\begin{theo}\label{theo:H2ExistBoundedMag}
With the same notation and same assumptions as in Lemma \ref{lemm:H11}, but now for any initial state $\varphi_0 \in H^2_2$, the equation 
\begin{equation*}\begin{split}
    i \partial _t \psi(x,t) &= ( -i\nabla -A(x))^2 \psi(x,t) + V(x,t) \psi(x,t)  \\
    \psi(x,0) &= \varphi_0(x)
\end{split}\end{equation*}
has a unique solution $\psi \in L^\infty(0,T;H^2_2)$, and this $\psi$ satisfies $\left \lVert \psi \right \rVert _{L^\infty(0,T;H^2_2)} \leq C_{T,\alpha,\rho} \lVert \varphi_0 \rVert _{H^2_2}.$
\end{theo}

\begin{proof}
We start by deriving a bound on $\lVert \psi_\epsilon \rVert_{H^2}$ which we will do in an extra Lemma:
\begin{lemm} \label{H^2 bound}
There exists $C$ independent of $\epsilon$ and $t$ such that 
\[ \lVert \psi_\epsilon \rVert_{H^2} \leq  C \lVert \psi_\epsilon \rVert_{H_2} + C \lVert \partial_t \psi_\epsilon \rVert_{L_2} \]
\end{lemm}
Using some of the previously established results we see that
\begin{proof}
Directly from the Schrödinger equation, we obtain
 \begin{equation*}
 \begin{split}
     \lVert \psi_\epsilon \rVert_{H^2} & \leq  \lVert \partial_t \psi_\epsilon \rVert_{2} + 2 \mathcal{A} \lVert D_{\epsilon} \psi_\epsilon \rVert_{2} + (\mathcal{A+A}^2) \lVert \psi_\epsilon \rVert_{2} + \rho( |u(t)| + 1) \lVert \psi_\epsilon \rVert_{H_1} + \lVert W_{\operatorname{sing}} \psi_\epsilon \rVert_{2} \\
     & \leq  \lVert \partial_t \psi_\epsilon \rVert_{2} + ( 2\mathcal{A} +\delta) \lVert \nabla \psi_\epsilon \rVert_{2} + (\mathcal{A+A}^2 + \rho|u(t)| + \rho + C_\delta) \lVert \psi_\epsilon \rVert_{H_1}.
      \end{split}
 \end{equation*}
However, from $\lVert \nabla \psi_\epsilon \rVert_{2}  \leq \sqrt { \lVert \psi_\epsilon \rVert_{2}\lVert \psi_\epsilon \rVert_{H^2}} \leq \frac{1}{\kappa} \lVert \psi_\epsilon \rVert_{2} + \kappa \lVert \psi_\epsilon \rVert_{H^2}$, we see that by choosing $\delta, \kappa$ to be sufficiently small, there exists $C$ independent of $\epsilon$ and $t$ such that $\lVert \psi_\epsilon \rVert_{H^2} \leq  C \lVert \psi_\epsilon \rVert_{H_2} + C \lVert \partial_t \psi_\epsilon \rVert_{L_2}.$

\end{proof}

We now continue with the energy estimates:
Starting from $i \partial_t \psi_\epsilon = H_B^\epsilon \psi_\epsilon + V^\epsilon \psi_\epsilon$ , multiply by $|x|^4 \overline{\psi_\epsilon}$, integrate over $\RR^3$ and take the imaginary part

\begin{equation*}
\begin{split}
    &\frac{1}{2} \frac{d}{dt} \int_{\RR^3} |x|^4 |\psi_\epsilon|^2 \ dx 
     = \Im \int_{\RR^3} (\nabla \psi_\epsilon) \cdot \nabla (|x|^4\psi_\epsilon) - |x|^4 \overline{\psi_\epsilon} (AD_{\epsilon} + D_{-\epsilon}A)\psi_\epsilon \ dx \\
    & \leq 2 \int_{\RR^3} |x|^2  |\nabla \psi_\epsilon|^2 \ dx + 2 \int_{\RR^3} |x|^4  |\psi_\epsilon|^2 \ dx - \Im \int_{\RR^3} |x|^4 \overline{\psi_\epsilon} (AD_{\epsilon} + D_{-\epsilon}A)\psi_\epsilon \ dx.
\end{split}
\end{equation*}
Note that
\begin{equation*}
\begin{split}
    \Im \int_{\RR^3} |x|^4 \overline{\psi_\epsilon} (AD_{\epsilon} + D_{-\epsilon}A)\psi_\epsilon \ dx
    &= \Im \int_{\RR^3} \overline{\psi_\epsilon} (|x|^4 AD_{\epsilon} + D_{-\epsilon}|x|^4 A)\psi_\epsilon \ dx \\
    & \quad + \Im \int_{\RR^3} \overline{\psi_\epsilon} (|x|^4 D_{-\epsilon}A - D_{-\epsilon}|x|^4 A)\psi_\epsilon \ dx.
\end{split}
\end{equation*}
The first term of the RHS is in $\RR$ by self-adjointness. Therefore

\begin{equation*}
\begin{split}
    &\Im \int_{\RR^3} |x|^4 \overline{\psi_\epsilon} (AD_{\epsilon} + D_{-\epsilon}A)\psi_\epsilon
    = \Im \int_{\RR^3} \overline{\psi_\epsilon} (|x|^4 D_{-\epsilon}A - D_{-\epsilon}|x|^4 A)\psi_\epsilon \ dx \\
    & \quad - \sum_j (|x - \epsilon e_j|^4 A_j(x-\epsilon e_j) \psi_\epsilon(x-\epsilon e_j) - |x|^4 A_j(x)\psi_\epsilon(x) )  \Bigg) \ dx \\
     &= \Im \int_{\RR^3} \overline{\psi_\epsilon} \frac{i}{\epsilon} \sum_j \left( (|x - \epsilon e_j|^4 - |x|^4) A_j(x-\epsilon e_j) \psi_\epsilon(x-\epsilon e_j) \right) \ dx \lesssim \mathcal{A} \int_{\RR^3} |x|^4 |\psi_\epsilon |^2 \ dx
     \end{split}
\end{equation*}
for some constant $C>0.$ Therefore, we have shown that 
\begin{equation*}
\begin{split}
    \frac{d}{dt} \int_{\RR^3} |x|^4 |\psi_\epsilon|^2 \ dx 
    & \leq 4 \int_{\RR^3} |x|^2  |\nabla \psi_\epsilon|^2 \ dx + (2C \mathcal{A}+4) \int_{\RR^3} |x|^4  |\psi_\epsilon|^2 \ dx.
\end{split}
\end{equation*}

Now notice that
\begin{equation*}
\begin{split}
     \int_{\RR^3} |x|^2  |\nabla \psi_\epsilon|^2 \ dx
     & = - \int_{\RR^3} \overline{\psi_\epsilon} \left( |x|^2  \Delta \psi_\epsilon + 2x \cdot \nabla \psi_\epsilon \right) \ dx\\
     & \leq \int_{\RR^3} |x|^2|\psi_\epsilon|^2 + |\nabla \psi_\epsilon|^2  +  |x|^4|\psi_\epsilon|^2 + |\partial_t \psi_\epsilon|^2  + |x|^2|\psi_\epsilon|^2 A\cdot A \ dx \\
     & \quad + \rho(|u(t)|+1)\int_{\RR^3} (|x|^2+|x|^4)|\psi_\epsilon|^2  \ dx + \int_{\RR^3}|x|^4|\psi_\epsilon|^2 + |\psi_\epsilon|^2|W_{\operatorname{sing}}| \ dx  \\
     & \quad + \lVert |\bullet|^2 \psi_\epsilon \rVert_{2} \lVert AD_{\epsilon} \psi_\epsilon \rVert_{2} + \lVert |\bullet|^2 \psi_\epsilon \rVert_{2} \lVert D_{-\epsilon}A \psi_\epsilon \rVert_{2} \\
     & \lesssim \lVert \psi_\epsilon \rVert_{H_2} ^2 +\lVert \psi_\epsilon \rVert_{H^1} ^2 + \lVert \partial_t \psi_\epsilon \rVert_{2} ^2.
\end{split}
\end{equation*}
By Lemma \ref{H^2 bound} and the conservation of $\lVert \psi_\epsilon \rVert_{2},$ we conclude that independent of $t$ and $\epsilon$
\begin{equation*}
    \frac{d}{dt} \int_{\RR^3} (1+|x|^4) |\psi_\epsilon|^2 \ dx  \lesssim  \lVert \psi_\epsilon \rVert_{H_2} ^2 +  \lVert \partial_t \psi_\epsilon \rVert_{2}^2.
\end{equation*}

Let $E_\epsilon(t) = \lVert \psi_\epsilon \rVert_{H_2} ^2 +  \lVert \partial_t \psi_\epsilon \rVert_{2} ^2 $.
By integrating the previous line from 0 to $t$, we see that
\begin{equation*}
    \int_{\RR^3} (1+|x|^4) |\psi_\epsilon(t)|^2 \ dx  \lesssim \int_0^t E_\epsilon(t) dt + \int_{\RR^3} (1+|x|^4) |\varphi_0|^2 \ dx.
\end{equation*}
Therefore, $\lVert \psi_\epsilon(t) \rVert_{H_2}^2 \lesssim  \int_0^t E_\epsilon(t) dt +  \lVert \varphi_0 \rVert_{H_2}^2.$ We now aim to find a bound on $\lVert \partial_t \psi_\epsilon \rVert_{2} ^2$ .

Let $\theta = \partial_t \psi_\epsilon$ then from $i \partial_t \theta  = H_B^\epsilon \theta + (W_{\operatorname{reg}}^\epsilon + W_{\operatorname{sing}}^\epsilon ) \theta + u'(t) V_{\operatorname{con}}^\epsilon \psi_\epsilon + u(t) V_{\operatorname{con}}^\epsilon \theta,$
we find that 
\[ \frac{d}{dt}   \lVert \theta \rVert_{2}^2     \leq |u'(t)| \rho \lVert \psi_\epsilon \rVert_{H_2}^2 +  |u'(t)| \lVert \theta \rVert_{2}^2.\]
Therefore there exists $C_\rho$ depending only on $\rho$ such that 
\begin{equation*}
\begin{split}
    \lVert \theta(t) \rVert_{2}^2  & \leq \lVert \theta(0) \rVert_{2}^2 + C_\rho \int_0^t |u'(s)|  E_\epsilon(s) ds
 \end{split}
\end{equation*}
and  $ \lVert \theta(0) \rVert_{2}  = \lVert H_B^\epsilon \varphi_0 + V^\epsilon \varphi_0 \rVert_{2} \lesssim \lVert \varphi_0 \rVert_{H^2_2}.$

By combining these estimates, we get that there exists $C$ independent of $t$ and $\epsilon$ such that
\[E_\epsilon(t) \lesssim \lVert \varphi_0 \rVert_{H^2_2}^2 + C \int_0^t (1+|u'(s)|) E_\epsilon(s) \ ds.\]
Therefore by Gronwall's inequality $E_\epsilon(t) \lesssim e^{\alpha + T} \lVert \varphi_0 \rVert_{H^2_2}^2.$
Using Lemma \ref{H^2 bound} we see we can bound $\lVert \psi_\epsilon(t) \rVert_{H^2_2}$ by a multiple of $\sqrt{E_\epsilon(t)}$ , therefore there exists $C$ independent of $t$ and $\epsilon$ such that $\lVert \psi_\epsilon(t) \rVert_{H^2_2} \lesssim  \lVert \varphi_0 \rVert_{H^2_2}.$
\end{proof}

We now treat the case of a homogeneous magnetic field:

\begin{theo}
\label{theo:magfield}
With the same notation and same assumptions as in Lemma \ref{lemm:H11} but for homogeneous magnetic fields with vector potential
$A(x) = \frac{B_0}{2}(-x_2,x_1,0),$ there exists a constant $C_{T,\alpha,\rho}$ such that for any $\varphi_0 \in H^2_2$, the equation
\begin{equation*}\begin{split}
    i \partial _t \psi(x,t) &= ( -i\nabla -A)^2 \psi(x,t) + V \psi(x,t)  \text{ with }\psi(\bullet,0) = \varphi_0
\end{split}\end{equation*}
has a unique solution $\psi \in L^\infty(0,T;H^2_2)$ satisfying $
    \left \lVert \psi \right \rVert _{L^\infty(0,T;H^2_2)} \leq C_{T,\alpha,\rho} \lVert \varphi_0 \rVert _{H^2_2}.$
\end{theo}

\begin{proof}
 
 Let $A_\epsilon(x):= \frac{B_0}{2\epsilon}(-\tanh(\epsilon x_2),\tanh(\epsilon x_1),0).$ We know from Theorem \ref{theo:H2ExistBoundedMag} that for all $\varphi_0 $ there exists $\psi_\epsilon \in H^2_2$ solving
 \begin{equation*}\begin{split}
    i \partial _t \psi_\epsilon(x,t) &= ( -i\nabla -A_\epsilon)^2 \psi_\epsilon(x,t) + V \psi_\epsilon(x,t)  \text{ with } \psi_\epsilon(x,0) = \varphi_0(x).
\end{split}\end{equation*}

Lemma \ref{H^2 bound} still holds for $\psi_\epsilon$ since
 \begin{equation}
 \begin{split}
 \label{eq:identity123}
   \lVert \Delta \psi_\epsilon \rVert_{2}&  \leq  \lVert \partial_t \psi_\epsilon \rVert_{2} + 2\lVert A_\epsilon \nabla \psi_\epsilon \rVert_{2} + \lVert A_\epsilon \cdot A_\epsilon \psi_\epsilon \rVert_{2} + \lVert V \psi_\epsilon \rVert_{2} 
 \end{split}
 \end{equation}
 and integration by parts implies for $\kappa>0$ arbitrary
\begin{equation*}
\begin{split}
\lVert A_\epsilon \nabla \psi_\epsilon \rVert_{2}^2  &\leq \int_{\RR^3}  \Big(|\nabla ( A_\epsilon \cdot A_\epsilon)| |\psi_\epsilon| |\nabla \psi_\epsilon| +  A_\epsilon \cdot A_\epsilon |\psi_\epsilon| |\Delta \psi_\epsilon|\Big)(x) dx \\
      &\leq \frac{B_0^4}{4\kappa} \lVert \psi_\epsilon \rVert_{H_1}^2 + \kappa  \lVert \psi_\epsilon \rVert_{H^1}^2 + \frac{B_0^2}{4\kappa} \lVert \psi_\epsilon \rVert_{H_2}^2 + \kappa  \lVert \psi_\epsilon \rVert_{H^2}^2.
 \end{split}
 \end{equation*}
 
  By picking $\kappa$ to be sufficiently small, it is clear in \eqref{eq:identity123} that $ \lVert \psi_\epsilon \rVert_{H^2} \lesssim  C \lVert \psi_\epsilon \rVert_{H_2} + C \lVert \partial_t \psi_\epsilon \rVert_{L_2}.$

Therefore there exists $C$ independent of $\epsilon$ and $t$ such that $\lVert \psi_\epsilon \rVert_{H^2_2} \leq  C \sqrt{E_\epsilon(t)}$ 
where $E_\epsilon(t) =  \lVert \psi_\epsilon \rVert_{H_2}^2 + \lVert \partial_t \psi_\epsilon \rVert_{L_2}^2.$
 
To bound $\lVert \psi_\epsilon \rVert_{H_2}$, we start from $ i \partial _t \psi_\epsilon = ( -i\nabla -A_\epsilon)^2 \psi_\epsilon + V \psi_\epsilon$ and multiply by $|x|^4 \overline{\psi_\epsilon}$. Taking the imaginary part and integrating over $\RR^3$, we see that
\begin{equation*}
\begin{split}
    \Re \int_{\RR^3} |x|^4 (\overline{\psi_\epsilon} \partial_t \psi_\epsilon)(x) \ dx & = \Im \int_{\RR^3} -|x|^4 (\overline{\psi_\epsilon} \Delta \psi_\epsilon)(x) - 2i |x|^4 (\overline{\psi_\epsilon} A_\epsilon \cdot \nabla \psi_\epsilon)(x) \ dx.
 \end{split}
 \end{equation*}
Integration by parts yields
\begin{equation*}
\begin{split}
    \Im \int_{\RR^3} 2i |x|^4 (\overline{\psi_\epsilon} A_\epsilon \cdot \nabla \psi_\epsilon)(x) \ dx &= \Re \int_{\RR^3}  - 2 \psi_\epsilon(x) \nabla \cdot ( |x|^4 \overline{\psi_\epsilon}(x) A_\epsilon(x))  \ dx\\
    &= \Im \int_{\RR^3}  - 2i |\psi_\epsilon(x)|^2 A_\epsilon(x) \cdot \nabla ( |x|^4 ) -2i \overline{\psi_\epsilon(x)}  |x|^4 A_\epsilon(x) \cdot \nabla \psi_\epsilon(x) \ dx \end{split}
 \end{equation*}
 which shows that
\[  \Im \int_{\RR^3} 2i |x|^4 (\overline{\psi_\epsilon} A_\epsilon \cdot \nabla \psi_\epsilon)(x) \ dx =\Im \int_{\RR^3}  - i |\psi_\epsilon(x)|^2 A_\epsilon(x) \cdot \nabla ( |x|^4 ) \ dx.\]
Hence, this implies
\begin{equation*}
\begin{split}
    \frac{1}{2} \frac{d}{dt} \int_{\RR^3} |x|^4  |\psi_\epsilon(x)|^2 \ dx & = \Im \int_{\RR^3} \nabla (|x|^4 \overline{\psi_\epsilon(x)}) \cdot ( \nabla \psi_\epsilon(x)) +i |\psi_\epsilon(x)|^2 A_\epsilon(x) \cdot \nabla ( |x|^4 )  \ dx\\
    &\leq 4 \int_{\RR^3} |\psi_\epsilon(x)| |x|^3  | \nabla \psi_\epsilon(x)| + |\psi_\epsilon(x)|^2 |A_\epsilon(x)|  |x|^3 \ dx\\
       &\lesssim (\vert B_0 \vert +1) \lVert \psi_\epsilon \rVert_{H_2}^2 + \lVert \psi_\epsilon \rVert_{H^2}^2.
\end{split}
\end{equation*}
Therefore we obtain the estimate $\lVert \psi_\epsilon \rVert_{H_2}^2 \lesssim   \lVert \varphi_0 \rVert_{H_2}^2 +  \int_0^t E_\epsilon(s) \ ds.$

By exactly the same argument as last time we see that there exists $C_\rho$ depending only on $\rho$ such that 
\begin{equation*}
\begin{split}
    \lVert  \partial_t \psi_\epsilon(t) \rVert_{2}^2  & \leq \lVert  \partial_t \psi_\epsilon(0) \rVert_{2}^2 + C_\rho \int_0^t |u'(s)|  E_\epsilon(s) ds.
 \end{split}
\end{equation*}
By combining these estimates, we get that there exists $C$ independent of $t$ and $\epsilon$ such that
\begin{equation*}
    E_\epsilon(t) \lesssim \lVert \varphi_0 \rVert_{H^2_2}^2 + C \int_0^t (1+|u'(s)|) E_\epsilon(s) \ ds.
\end{equation*}
Therefore by Gronwall's inequality $E_\epsilon(t) \lesssim e^{\alpha + T} \lVert \varphi_0 \rVert_{H^2_2}^2$ and thus $\lVert \psi_\epsilon(t) \rVert_{H^2_2} \lesssim  \lVert \varphi_0 \rVert_{H^2_2}.$
\end{proof}

Our next Lemma will be used to obtain an explicit decay estimate on the full solution to the Schrödinger equation with magnetic field on a bounded domain:
\begin{lemm}
\label{lemm:decay}
There exists $C>0$ such that for all $R\ge r>2$ and $f \in H^2_2(B_R(0)) \cap H^1_0(B_R(0))$,
\[
 \left| \int_{\partial B_r(0)} \overline{f(x)} \nabla_n f(x) \ dS(x) \right|  \lesssim \frac{ \lVert f \rVert_{H^2_2}^2}{r^2}.
\]
\end{lemm}

\begin{proof}
We introduce the notation $A_{r_1,r_2} := B_{r_2}(0) \setminus B_{r_1}(0) $ and $A_r := A_{r,\infty}$. For all $R\ge r>2$:
\begin{equation}
\begin{split}
    \inf_{r \in [s-1,s]} \int_{\partial B_r(0)} |f(x) \nabla_n f(x)| \ dS(x) & \leq \int_{s-1}^{s} \int_{\partial B_r(0)} |f(x) \nabla_n f(x)| \ dS(x) \ dr \\
    &\leq \int_{A_{s-1,s}} |f(x) \nabla f(x)| \ dx \\
    &\leq  \left( \int_{A_{s-1,s}} |f(x)|^2 \ dx \int_{A_{s-1,s}} |\nabla f(x)|^2 \ dx \right)^\frac{1}{2} \\
    &\leq  \frac{\lVert f \rVert_{H_2(B_{s})} \lVert f \rVert_{H^1(B_{s})}}{(s-1)^2}.
\end{split}
\end{equation}  
Since $h(s) = \int_{\partial B_s(0)} |f(x) \nabla_n f(x)| \ dS(x) $ is continuous, there exists $r_s \in [s-1,s]$ at which this infimum is attained.
Now notice that
\begin{equation}
\begin{split}
\label{eq:aux_eq}
    \left| \int_{\partial B_{s}(0)} \overline{f(x)} \nabla_n f(x) \ dS(x) \right| & \leq  
    \left| \int_{A_{r_s,s}} (\bar{f} \Delta f)(x) \ dx \right| + \left| \int_{A_{r_s,s}} |\nabla f(x)|^2 \ dx \right| \\
    &\qquad + \left| \int_{\partial B_{r_s}(0)} \bar{f}(x) \nabla_n f(x) \ dS(x) \right| \\
    &\leq  \int_{A_{r_s,s}} \left| (\bar{f} \Delta f)(x) \right|\ dx  + \int_{A_{r_s,s}} |\nabla f(x)|^2 \ dx \\
    &+ (s-1)^{-2} \lVert f \rVert_{H_2(B_{s})} \lVert f \rVert_{H^1(B_{s})}.
\end{split}
\end{equation}
It now remains to estimate the two last integrals in \eqref{eq:aux_eq}.
    We see that
\begin{equation}
\label{eq:zfw1}
\begin{split}    
     \left| \int_{A_{r_s,s}} (\bar{f} \Delta f)(x) \ dx \right| & \leq \left( \int_{A_{r_s,s}} |f(x)|^2 \ dx \int_{A_{r_s,s}} |\Delta f(x)|^2 \ dx \right)^\frac{1}{2} \\
     & \leq \frac{1}{(s-1)^2}\int_{A_{r_s,s}} (s-1)^4|f(x)|^2 \ dx + \frac{1}{(s-1)^2} \int_{A_{r_s,s}} |\Delta f(x)|^2 \ dx \\
     & \leq \frac{1}{(s-1)^2} \lVert f \rVert^2_{H^2_2(B_{s})}.
\end{split}
\end{equation}    
Here, we used that $f$ vanishes on $B_R(0)$ by the Dirichlet boundary condition and estimates \eqref{eq:aux_eq} and \eqref{eq:zfw1}.
    Therefore, we find since $n(x)=\frac{x}{\vert x \vert}$
\begin{equation}
\begin{split}
    \left| \int_{\partial B_r(0)} (|\psi|^2 A \cdot n)(x) \ dS(x) \right| &\leq \frac{\rho}{r} \int_{\partial B_r(0)} |\psi(x)|^2 x \cdot n \ dS(x) = \frac{\rho}{r} \int_{ B_r(0)} \nabla \cdot (|\psi(x)|^2 x) \ dx \\
    &\leq \frac{\rho}{r} \left(   3 r^{-4}  \int_{ B_r(0)}  |x|^4|\psi(x)|^2  \ dx + 2\int_{ B_r(0)} |x| |\psi(x)|  |\nabla \psi(x)| \ dx    \right) \\
    &\leq \frac{\rho}{r} \Bigg(   3 r^{-4}  \int_{ B_r(0)}  |x|^4|\psi(x)|^2  \ dx + r^{-1} \int_{ B_r(0)} |x|^4 |\psi(x)|^2  \ dx \\
    &+ r^{-1} \int_{ B_r(0)}  |\nabla \psi(x)|^2 \ dx    \Bigg) \leq r^{-2} C \lVert \psi \rVert^2_{H^2_2}.
\end{split}    
\end{equation}

Therefore there exists $C>0$ such that for all $R\ge s>2$, $f \in H^2_2$
\begin{equation}
\begin{split} 
\left| \int_{\partial B_{s}(0)} \bar{f}(x) \nabla_n f(x) \ dS(x) \right| \leq \frac{\frac{1}{4}C}{(s-1)^2} \lVert f \rVert^2_{H^2_2(B_{s})} \leq \frac{C}{s^2} \lVert f \rVert^2_{H^2_2(B_{s})}
\end{split}
\end{equation}
as required.
\end{proof}

We now want to compare the solution $\psi$, where $\psi$ solves the equation $ i \partial_t \psi = (H_0^B +V_{\operatorname{TD}})\psi $ on $\RR^3$ and has initial condition $\varphi_0 \in H^2_2$, to a solution $\psi^D$, where $\psi^D$ solves the same equation on a bounded domain $\Omega$ with Dirichlet boundary conditions and initial condition $\varphi^D_0 \in H^2(\Omega) \cap H^1_0(\Omega) $.

It is easy to check that the existence of such a solution on the bounded domain $\Omega$ follows as in Theorem \ref{theo:H2ExistBoundedMag}. The essential ingredient is the existence of a self-adjoint magnetic Dirichlet Laplacian with domain $H^2_2(\Omega) \cap H^1_0(\Omega)$ defined by the quadratic form 
$$ q: H^1_0(\Omega)^2 \rightarrow \RR \text{ with }q(\psi,\psi)=\langle (-i\nabla-A) \psi,(-i\nabla-A) \psi \rangle_{L^2}.$$
\begin{lemm}
\label{lemm:aux}
With the same notation and constants as in Theorem \ref{theo:H2ExistBoundedMag}, but now for any initial state $\varphi_0 \in H^2_2(\Omega)\cap H^1_0(\Omega)$ on $\Omega$ with Dirichlet boundary conditions, the equation 
\begin{equation}\begin{split}
\label{eq:Dirichlet}
    i \partial _t \psi(x,t) &= ( -i\nabla -A)^2 \psi(x,t) + V \psi(x,t)  \\
    \psi(x,0) &= \varphi_0(x)
\end{split}\end{equation}
has a unique solution $\psi \in L^\infty(0,T;H^2_2)$, and this $\psi$ satisfies $\left \lVert \psi \right \rVert _{L^\infty(0,T;H^2_2)} \leq C_{T,\alpha,\rho} \lVert \varphi_0 \rVert _{H^2_2}.$

Similarly, under the same notation and constants as in Lemma \ref{lemm:H11}, we have for the above equation $\left \lVert \psi \right \rVert _{L^\infty(0,T;H^2_2)} \leq C_{T,\alpha,\rho,\mathcal A} \lVert \varphi_0 \rVert _{H^2_2}.$
\end{lemm} 
The difference $\xi = \psi - \psi^D$ on $\Omega$ solves $i \partial_t \xi = (H_0^B +V_{\operatorname{TD}})\xi.$

Multiplying by $\bar \xi$, taking imaginary parts and integration by parts gives for $n$ the unit normal
\begin{equation*}
\begin{split}
\frac{1}{2} \frac{d}{dt} \int_\Omega |\xi(x)|^2 \ dx &= - \Im \left( \int_{\partial \Omega} (\overline{\xi} \nabla_n \xi)(x) \ dS \right) + \int_{\partial \Omega} |\psi(x)|^2 (A \cdot n)(x) \ dS.
\end{split}
\end{equation*}

If $\Omega = B_R(0)$ and $A = \frac{B_0}{2}(-x_2,x_1,0),$ the potential of a homogeneous field, then $A \cdot n = 0.$ While instead if $A$ is a bounded function with $ \lVert A \rVert_\infty \leq \rho $, then

\begin{equation*}
\begin{split}
    \left| \int_{\partial B_R(0)} |\psi(x)|^2 (A \cdot n)(x) \ dS(x) \right| &= \left\lvert  \int_{ B_R(0)} \nabla \cdot (|\psi|^2 A)(x) \ dx \right\rvert \le \left\lvert \rho \int_{ B_R(0)} |\psi(x)|^2+2 \vert (\psi \nabla \psi)(x) \vert  \ dx \right\rvert \\
    & \leq \frac{\rho}{R} \left(   3 R^{-4}  \int_{ B_R(0)}  |x|^4|\psi(x)|^2  \ dx + 2\int_{ B_R(0)} |x| |\psi(x)|  |\nabla \psi(x)| \ dx    \right) \\
    &\leq \frac{\rho}{R} \Bigg(   3 R^{-4}  \int_{ B_R(0)}  |x|^4|\psi(x)|^2  \ dx + R^{-1} \int_{ B_R(0)} |x|^4 |\psi(x)|^2  \ dx \\
    &+ R^{-1} \int_{ B_R(0)}  |\nabla \psi(x)|^2 \ dx    \Bigg) \lesssim R^{-2}  \lVert \psi \rVert^2_{H^2_2}.
\end{split}    
\end{equation*}

Now notice by Lemma \ref{lemm:decay}
\begin{equation*}
\begin{split}
    \int_{\partial B_R(0)} \overline{\xi(x)} \nabla_n \xi(x) \ dS(x) &\lesssim \frac{ \lVert \psi - \psi^D \rVert^2_{H^2_2}}{R^2}\lesssim \frac{\lVert \varphi_0 \rVert^2_{H^2_2} + \lVert \varphi_0^D \rVert^2_{H^2_2} }{R^2}.
\end{split}    
\end{equation*}

Thus, we have proven the following Theorem
\begin{theo}
\label{theo:Dirichlet}
Let $\psi$ be the solution to the linear Schrödinger equation with magnetic field \eqref{eq:bilSchr} with potentials and magnetic vector potentials satisfying Assumption \ref{ass:ass1}. Then, there is some absolute constant $C_{\rho,T}$ such that for any $R$ sufficiently large, where $\psi^D$ is the solution to the Dirichlet problem \eqref{eq:Dirichlet}
 on $B_R(0),$
\[ \Vert \psi-\psi^D \Vert_{2} \le C_{\rho,T}/R^2.\]
\end{theo}

\section{Hartree equation}
\label{sec:Hartree}
We now turn to the evolution of an effective single particle described by the Hartree equation with Schr\"odinger operator $H_0=(-i\nabla-A)^2+V: D(H_0) \subset L^2(\mathbb R^3) \rightarrow L^2(\mathbb R^3)$, static \emph{pinning potential} $V$, \emph{control potential} $V_{\operatorname{con}},$ and time-dependent control function $u \in W^{1,1}(0,T)$ such that $V_{\operatorname{TD}}(t):= u(t) V_{\operatorname{con}}:$
\begin{equation}
\begin{split}
\label{eq:HF}
i \partial_t \psi(x,t) &= (H_0 + V_{\operatorname{TD}}(t))\psi(x,t)+\left(\vert \psi(\bullet,t) \vert^2*\frac{1}{\vert \bullet \vert} \right)(x)\psi(x,t), \quad (x,t) \in \mathbb R^3 \times (0,T) \\
\psi(\bullet,0)&=\varphi_0.
\end{split}
\end{equation}

under Assumption \ref{ass:ass1} on the potentials.

\begin{rem}
Since we discussed the magnetic field and its complications already in the first part of this article, we shall neglect it in the treatment of the Hartree equation in the sequel to simplify the presentation.
\end{rem}

\subsection{Local Existence of Solutions to the Hartree equation}

We first need to show that equation \eqref{eq:HF} has a unique solution on $\RR^3\times [0,T]$. To do this we start by showing that a unique local solution exists and use an energy estimate to show the local solution can be extended into a unique global solution.

To show the existence of the local solution we collect in the following Lemma some basic estimates on the Hartree non-linearity, as in \cite[Lemma $2.3$]{B}, that we shall frequently use throughout this section:
\begin{lemm}
\label{EuanRulesForF}

Let $\Omega \subset \mathbb R^3$ be a domain and $\psi\in H^1(\Omega)$, we then define $F_{\Omega}(\psi)(x)=\int_{\Omega} \frac{|\psi(x)|^2}{|x-y|} \ dy \ \psi(x)$ and just write $F:=F_{\mathbb R^3}.$
Then, there are constants $c,C_F>0$ independent of $\Omega$ such that 
\begin{enumerate}
    \item For all $\psi,\phi\in H^1(\Omega)$ we have $||F_{\Omega}(\psi)-F_{\Omega}(\phi)||_{2}\le c (||\psi||^2_{H^1}+||\phi||^2_{H^1})||\psi-\phi||_{2}.$
        \item There is $C_F>0$ such that for all $\psi,\phi\in H^2_2(\Omega)$
    \begin{equation*}\begin{split}
        ||F_{\Omega}(\psi)-F_{\Omega}(\phi)||_{H^2_2}&\leq C_F(||\psi||^2_{H^1}+||\phi||^2_{H^2_2})||\psi-\phi||_{H^2_2}\text{ and }\\
        ||F_{\Omega}(\psi)||_{H^2_2}&\leq C_F||\psi||^2_{H^1}||\psi||_{H^2_2}.
    \end{split}\end{equation*}
\end{enumerate}
\end{lemm}

For the linear part of the nonlinear evolution equation \eqref{eq:HF}, Theorem \ref{theo:H2ExistBoundedMag} and Lemma \ref{lemm:aux} guarantee the existence of evolution operators associated with the time-dependent Hamiltonian
\[    H(t) = (-i \nabla-A)^2+W_{\operatorname{reg}}+W_{\operatorname{sing}}+u(t)V_{\operatorname{con}}\]
which is a family $\{U(t,s):t,s\in [0,T]\}$ on $H^2_2(\Omega) \cap H^1_0(\Omega)$ such that for all $\psi_0\in H^2_2(\Omega) \cap H^1_0(\Omega)$ the following properties hold
\begin{enumerate}
    \item $U(t,s)U(s,r)\psi_0=U(t,r)\psi_0$ and $U(t,t)\psi_0=\psi_0,$
    \item $(t,s)\mapsto U(t,s)\psi_0$ is strongly continuous in $L^2$ on $[0,T]^2$ and $U(t,s)$ is an isometry on $L^2(\Omega),$
    \item For all $ s,t\in [0,T], U(t,s)\in \mathcal L(H^2_2)$ and $(t,s)\mapsto U(t,s)\psi_0$ is weakly continuous from $[0,T]^2$ to $H^2_2(\Omega)$. Moreover there is 
    \[M=M(T,||u||_W^{1,1},||W_{\operatorname{sing}}||_{L^p},||\langle\bullet\rangle^{-2}W_{\operatorname{reg}}||_{L^\infty},||\langle\bullet\rangle^{-2}V_{\operatorname{con}}||_{L^\infty})\]
    
    such that $||U(t,s)\psi||_{H^2_2}\leq M||\psi||_{H^2_2}.$
    \item In the $L^2(\Omega)$-sense, we have $i\partial_t U(t,s)\psi_0=H(t)U(t,s)\psi_0$ and $i\partial_s U(t,s)\psi_0=U(t,s)H(t)\psi_0.$
\end{enumerate}
The evolution operators exist by Theorem \ref{theo:H2ExistBoundedMag}.

Now we can prove the existence and uniqueness of a local solution using a simple contraction argument on $\mathbb R^3:$
\begin{lemm}
\label{EuanExistenceLocalSoln}
For $T_0>0$ small enough, there exists a unique solution in $H^2_2(\Omega)$ to
\begin{equation*}
\begin{split}
i \partial_t \psi(x,t) &= (H_0 + V_{\operatorname{TD}}(t))\psi(x,t)+\left(\vert \psi(\bullet,t) \vert^2*\frac{1}{\vert \bullet \vert} \right)(x)\psi(x,t), \quad (x,t) \in \mathbb R^3 \times (0,T_0) \\
\psi(\bullet,0)&=\varphi_0.
\end{split}
\end{equation*}

\end{lemm}

\begin{proof}

Consider the functional $\Lambda:\psi \mapsto U(\bullet,0)\psi_0-i\int_0^\bullet U(\bullet,s)F(\psi(s))ds $

and the set $
    B = \{\psi \in L^\infty((0,\tau);H^2_2),||\psi||_{L^\infty(0,\tau;H^2_2)}\leq 2M||\psi_0||_{H^2_2}\}.$

We want $\tau$ to be sufficiently small that $\Lambda$ is a contraction on $B$, i.e. for $\psi\in B$ and Lemma \ref{EuanRulesForF}
\begin{align*}
    ||\Lambda(\psi)(t)||_{H^2_2} &= \left\lVert U(t,0)\psi_0-i\int_0^t U(t,s)F(\psi(s))ds \right\rVert_{H^2_2}\\
    &\leq M||\psi_0||_{H^2_2}+8\tau M^4C_F ||\psi_0||_{H^2_2}.
\end{align*}

So we need to choose $\tau>0$ sufficiently small that $8\tau C_F M^3||\psi_0||^2_{H^2_2}<1$, and then
\begin{equation*}\begin{split}
    ||\Lambda(\psi)||_{L^\infty(0,\tau;H^2_2)}\leq 2M||\psi_0||_{H^2_2}.
\end{split}\end{equation*}

So $\Lambda(\psi)\in B$, and thus $\Lambda$ maps $B$ into itself. Now, for $\psi,\phi\in B$ then for all times $ t\in[0,\tau]$, using again Lemma \ref{EuanRulesForF}
\begin{align*}
    ||\Lambda(\psi)(t)-\Lambda(\phi)(t)||_{H^2_2}&=\left\lVert \int_0^t U(t,s)(F(\psi(s))-F(\phi(s)))\ ds\right\rVert_{H^2_2}\\
    &\leq 8\tau C_F M^3 ||\psi_0||^2_{H^2_2}||\psi-\phi||_{L^\infty([0,\tau];H^2_2)}.
\end{align*}

So as $8\tau C_F M^3 ||\psi_0||^2_{H^2_2}<1$, $\Lambda$ is a contraction on $B$ and so there is a unique solution to
\begin{equation*}\begin{split}
    \psi(t) = U(t,0)\psi_0-i\int_0^t U(t,s)F(\psi(s))ds.
\end{split}\end{equation*}
 So if we apply $\partial_t$ to this equation, and using that $i\partial_t U(t,s)\psi_0=H(t)U(t,s)\psi_0$ we get
\begin{equation*}\begin{split}
    \partial_t \psi &= i\Delta \psi - iW_{\operatorname{reg}}\psi - iW_{\operatorname{sing}}\psi - iu(t)V_{\operatorname{con}}\psi-iF(\psi(t))
\end{split}\end{equation*}
and so $\psi\in L^\infty((0,\tau);H^2_2)$ means that $F(\psi)\in L^\infty((0,\tau);H^2_2)$ and $\Delta\psi\in L^\infty((0,\tau);L^2)$.
We then recall the elementary bounds, which follow from Assumption \ref{ass:ass1},
\begin{equation*}\begin{split}
    ||W_{\operatorname{reg}}\psi(t)||_{2}&\leq \left\lVert \frac{W_{\operatorname{reg}}}{1+|\bullet|^2} \right\rVert_{L^\infty}||\psi(t)||_{H_2}\text{ and }
||u(t)V_{\operatorname{con}}\psi(t)||_{2}\leq \left\lVert\frac{ u(t)V_{\operatorname{con}}}{1+|\bullet|^2}\right\rVert_{L^\infty}||\psi(t)||_{H_2}.
\end{split}\end{equation*}

Now using \cite[Remark 3.8]{BH} $W_{\operatorname{sing}}\in L^2(\mathbb{R}^3)$ is infinitesimally bounded with respect to the negative Laplacian. Thus, for any $\epsilon > 0$ there is $C_\epsilon >0$ such that 
\begin{equation*}\begin{split}
    ||W_{\operatorname{sing}}\psi(t)||_{2}\leq \epsilon || -\Delta \psi(t)||_{2}+C_\epsilon ||\psi(t)||_{2}.
\end{split}\end{equation*}

This implies that $||W_{\operatorname{sing}}\psi(t)||_{2}\lesssim ||\psi(t)||_{H^2_2}.$
By combining the previous bounds, we find $||\partial_t \psi(t)||_{2} \leq C_2||\psi_0||_{H^2_2}.$

So it remains to prove that the solution is unique. Let $\psi$ and $\psi'$ be two solutions and $\phi=\psi-\psi'$. Then $\phi(0)=0$ and by subtracting the Schr\"odinger equation for each of the solutions, multiplying by $\overline{\phi}$, integrating over $\RR^3$ and taking the imaginary part we get
\begin{equation*}\begin{split}
        \Im\int_{\RR^3} i (\partial_t \phi) \overline{\phi}\ dx &=  \Im\int_{\RR^3}\left((-\Delta+W_{\operatorname{sing}}+W_{\operatorname{reg}}+u(t)V_{\operatorname{con}})\phi + F(\psi)-F(\psi')\right)\overline{\phi} \ dx.
        \end{split}\end{equation*}
        This shows that 
        \begin{equation*}\begin{split}
        \frac{1}{2}\frac{d}{dt}||\phi||^2_{2} &\lesssim  ||\phi||_{2}(||\psi||^2_{H^1}+||\psi'||^2_{H^1})||\psi-\psi'||_{2}\lesssim  ||\phi||^2_{2}.
\end{split}\end{equation*}

Then by Gronwall's inequality $||\phi||^2_{2}=0$ so $\psi=\psi'$ in $L^2$ and so the solution is unique.
\end{proof}

\subsection{Energy Estimate}

To extend our local solution to a global one we need an energy estimate of the solution of the Hartree equation for any arbitrary time $T$ and show that the energy is bounded on our interval $[0,T]$.

Multiplying the equation by $\partial_t\overline{\psi}$, integrating over $\mathbb{R}^3$, and taking the real part shows
\[ \Re \frac{d}{dt}\int_{\mathbb{R}^3} |\nabla\psi|^2  \ dx+\Re\int_{\mathbb{R}^3}\frac{1}{2}(W_{\operatorname{sing}}+W_{\operatorname{reg}}+u(t)V_{\operatorname{con}})\partial_t(|\psi|^2) + F(\psi)\partial_t\overline{\psi} \ dx=0.\]

Now to deal with the last term, we calculate
\begin{equation*}
    \begin{split}
        \frac{d}{dt}\int_{\RR^3} F(\psi)\overline{\psi} \ dx
        &=\int_{\RR^3} \left(|\psi|^2 * \frac{1}{|\bullet|}\right)(y) \partial_t |\psi(y,t)|^2\ dy+\int_{\RR^3} \left(|\psi|^2 * \frac{1}{|\bullet|}\right)(x) \partial_t |\psi(x,t)|^2 \ dx\\
        &= 4\Re \int_{\RR^3} \left(|\psi|^2 * \frac{1}{|\bullet|}\right)(x) \psi(x,t)\partial_t \overline{\psi(x,t)} \ dx.
    \end{split}
\end{equation*}

From this we conclude that
\begin{equation*}
    \begin{split}
        2\frac{d}{dt}\int_{\mathbb{R}^3} |\nabla\psi(x)|^2 \ dx  +2\int_{\mathbb{R}^3}(W_{\operatorname{sing}}+W_{\operatorname{reg}}+u(t)V_{\operatorname{con}})\partial_t(|\psi|^2) \ dx  + \frac{d}{dt}\int_{\mathbb{R}^3} F(\psi)\overline{\psi} \ dx=0.
    \end{split}
\end{equation*}

This in turn leads to 
\begin{equation*}
    \begin{split}
        \frac{d}{dt}\int_{\mathbb{R}^3}\left( |\nabla\psi|^2 +(W_{\operatorname{sing}}+W_{\operatorname{reg}}+u(t)V_{\operatorname{con}})|\psi|^2 + \frac{1}{2} F(\psi)\overline{\psi}\right) \ dx  &= \int_{\mathbb{R}^3} u'(t) V_{\operatorname{con}} |\psi|^2 \ dx \\
        &\leq  \left\Vert \frac{ u'(t)V_{\operatorname{con}}}{1+|\bullet|^2}\right\rVert_{\infty}||\psi||_{H_1}^2.
    \end{split}
\end{equation*}

Now by taking \eqref{eq:HF} and multiplying it by $(1+|x|^2)\overline{\psi}$, we can take the imaginary part to get
\begin{equation*}
    \begin{split}
        \Im i\partial_t\psi (1+|x|^2)\overline{\psi}&=-\Im(1+|x|^2)\overline{\psi}\Delta\psi\text{ and }
        (1+|x|^2)\partial_t(|\psi|^2) =  -\Im(1+|x|^2)\overline{\psi}\Delta\psi.
    \end{split}
\end{equation*}
Integrating over $\mathbb{R}^3$, we thus get
\begin{equation*}
    \begin{split}
        \frac{d}{dt}\int_{\mathbb{R}^3}(1+|x|^2)|\psi(x)|^2 \ dx &\lesssim \int_{\mathbb{R}^3}|\nabla\psi(x)|^2 \ dx +\int_{\mathbb{R}^3}|x|^2|\psi(x)|^2 \ dx.
    \end{split}
\end{equation*}

So now we can define $E$ at time $t\in[0,T]$ by
\begin{equation*}
    \begin{split}
        E(t) := \int_{\mathbb{R}^3}|\nabla \psi(x,t)|^2 + \lambda(1+|x|^2)|\psi(x,t)|^2+\frac{1}{2}\left(|\psi(t)|^2*\frac{1}{|\bullet|}\right)(x)|\psi(x,t)|^2 \ dx.
    \end{split}
\end{equation*}

Here, $\lambda>0$ is a constant to be determined later in such a way to allow $E(t)$ to be bounded
\begin{equation*}
    \begin{split}
        E'(t)&\leq  \left\lVert \frac{u'(t)V_{\operatorname{con}}}{1+|\bullet|^2}\right\rVert_{\infty}||\psi||_{H_1}^2+\frac{d}{dt}\int_{\mathbb{R}^3}(W_{\operatorname{sing}}+W_{\operatorname{reg}}+u(t)V_{\operatorname{con}})(x)|\psi(x,t)|^2 \ dx\\
        &\quad +\lambda C \left(\int_{\mathbb{R}^3}|\nabla\psi(x,t)|^2 \ dx+\int_{\mathbb{R}^3}|x|^2|\psi(x,t)|^2 \ dx\right) \\
        &\leq \left( \left\lVert \frac{u'(t) V_{\operatorname{con}}}{1+|\bullet|^2}\right\rVert_{\infty}+\lambda C\right)E(t)+\frac{d}{dt}\int_{\mathbb{R}^3}(W_{\operatorname{sing}}+W_{\operatorname{reg}}+u(t)V_{\operatorname{con}})(x)|\psi(x,t)|^2 \ dx.
    \end{split}
\end{equation*}

Integrating over $(0,t)$ we get
\begin{equation}
    \begin{split}
    \label{eq:energy}
        E(t) &\leq \int_0^t \left( \left\lVert \frac{u'(s)V_{\operatorname{con}}}{1+|\bullet|^2}\right\rVert_{\infty}+\lambda C\right)E(s) \ ds\\
        &+\int_{\mathbb{R}^3}(W_{\operatorname{sing}}(x)+W_{\operatorname{reg}}(x)+ u(t)V_{\operatorname{con}}(x)|\psi(x,t)|^2-u(0)V_{\operatorname{con}}(x)|\psi(x,0)|^2) \ dx.
    \end{split}
\end{equation}

Now we use again the infinitesimal boundedness with respect to the negative Laplacian, together with \cite[Theorem X.18]{RS} to find that for all $\epsilon>0$
\begin{equation*}
    \begin{split}
        |\langle\psi,W_{\operatorname{sing}}\psi\rangle_{2}|&\leq \epsilon |\langle\psi,\Delta\psi\rangle_{2}|+C_\epsilon\langle\psi,\psi\rangle_{2}=\epsilon ||\nabla\psi||_{2}^2+C_\epsilon ||\psi||_{2}^2.
    \end{split}
\end{equation*}

From this and \eqref{eq:energy}, we obtain the estimate
\begin{equation*}
    \begin{split}
        E(t) & \leq \int_0^t \left( \left\lVert\frac{ u'(s)V_{\operatorname{con}}}{1+|\bullet|^2}\right\rVert_{\infty}+\lambda C\right)E(s) \ ds+\epsilon ||\nabla \psi(0)||_{2}^2+C_\epsilon||\psi(0)||_{2}^2\\
        &+\left\lVert\frac{W_{\operatorname{reg}}}{1+|\bullet|^2}\right\rVert_{L^\infty}||\psi(0)||_{H_1}^2+\left\lVert\frac{u(0)V_{\operatorname{con}}}{1+|\bullet|^2}\right\rVert_{L^\infty}||\psi(0)||_{H_1}^2+\epsilon ||\nabla \psi(t)||_{2}^2\\
        &+C_\epsilon||\psi(t)||_{2}^2+\left\lVert\frac{W_{\operatorname{reg}}}{1+|\bullet|^2}\right\rVert_{L^\infty}||\psi(t)||_{H_1}^2+\left\Vert\frac{u(t)V_{\operatorname{con}}}{1+|\bullet|^2}\right\rVert_{L^\infty}||\psi(t)||_{H_1}^2.
    \end{split}
\end{equation*}

Taking $\epsilon=\frac{1}{2}$ we find by the preservation of the $L^2$ norm of $\psi$, and since $u\in W_{\operatorname{pcw}}^{1,1} [0,T]$ that
\begin{equation}
    \begin{split}
    \label{eq:Et}
        E(t) & \leq C + \int_0^t \left( \left\lVert \frac{u'(s)V_{\operatorname{con}}}{1+|\bullet|^2}\right\rVert_{\infty}+\lambda C\right)E(s) \ ds+\frac{||\nabla \psi(t)||_{2}^2}{2} +C'||\psi(t)||_{H_1}^2.
    \end{split}
\end{equation}

Noting that $E(0)\lesssim||\psi_0||^2_{H^1\cap H_1}+||\psi_0||_{H^1}||\psi_0||_{2}^3,$ we define $G$ for times $t\in [0,T]$ by
\begin{equation*}
    \begin{split}
        G(t)&:=\int_{\mathbb{R}^3}|\nabla\psi(t,x)|^2dx+\int_{\mathbb{R}^3}(1+|x|^2)|\psi(t,x)|^2dx+\int_{\mathbb{R}^3}\left(|\psi(t,\bullet)|^2*\frac{1}{|\bullet|}\right)(x)|\psi(t,x)|^2dx
    \end{split}
\end{equation*}

then taking $\lambda$ to be  $C'+\frac{1}{2}$ we can subtract $\frac{1}{2} ||\nabla \psi(t)||_{2}^2+C'||\psi(t)||_{H_1}^2$ from \eqref{eq:Et} to see that
\begin{equation*}
    \begin{split}
        G(t) \leq C + 2\int_0^t \left( \left\lVert \frac{u'(s)V_{\operatorname{con}}}{1+|\bullet|^2}\right\rVert_{\infty}+\lambda C\right)G(s) \ ds.
    \end{split}
\end{equation*}

Therefore, by Gronwall's inequality $G(t) \lesssim \exp \left( \int_0^t \beta(s)ds\right),$ where $\beta(t) =  \left\lVert \frac{u'(t)V_{\operatorname{con}}}{1+|\bullet|^2}\right\rVert_{\infty},$
and so there exists $C_{T,\psi_0}>0$ such that for all $t\in[0,T]$:
\begin{equation*}
    \begin{split}
        ||\psi(t)||^2_{H^1\cap H_1}+\int_{\mathbb{R}^3}\left(|\psi(x,t)|^2*\frac{1}{|\bullet|}\right)|\psi(x,t)|^2 \ dx \le C_{T,\psi_0}.
    \end{split}
\end{equation*}

\subsection{Global Existence}

Now that we have an estimate for the energy we can use that the equation is equivalent to the integral equation
\begin{equation*}
    \begin{split}
        \psi(t) = U(t,0)\psi_0-i\int_0^t U(t,s)F(\psi(s)) \ ds.
    \end{split}
\end{equation*}

Thus, the claim follows from Gr\"onwall's inequality applied to
\begin{equation*}
    \begin{split}
        ||\psi(t)||_{H^2_2} &\lesssim ||\psi_0||_{H^2_2}+\int_0^t ||F(\psi(s))||_{H^2_2} \ ds\\
        &\lesssim  ||\psi_0||_{H^2_2}+\int_0^t ||\psi(s)||_{H^1}^2||\psi(s)||_{H^2_2}\ ds \lesssim 1+\int_0^t ||\psi(s)||_{H^2_2} \ ds.
    \end{split}
\end{equation*}
Thus, we have shown that
\begin{theo}
\label{theo:HF}
With the same notation and same assumptions as in Lemma \ref{lemm:H11}
there exists a constant $C_{T}$ such that for any $\varphi_0 \in H^2_2$, equation \eqref{eq:HF}
\begin{equation*}\begin{split}
    i \partial _t \psi(x,t) &= ( -i\nabla -A)^2 \psi(x,t) + V \psi(x,t)  \text{ with }\psi(\bullet,0) = \varphi_0
\end{split}\end{equation*}
has a unique solution $\psi \in L^\infty(0,T;H^2_2)$ satisfying $
    \left \lVert \psi \right \rVert _{L^\infty(0,T;H^2_2)} \leq C_{T,\alpha,\rho} \lVert \varphi_0 \rVert _{H^2_2}.$
\end{theo}

\subsection{Existence and uniqueness of a solution on a ball}
We now look on a bounded open domain with piecewise smooth boundary $\Omega \subset \RR^3.$ 

To approximate the global dynamics on a bounded domain, we consider now the operator $H_0=-\Delta+V: D(H_0) \subset L^2(\Omega) \rightarrow L^2(\Omega)$ and take potentials as before, but now restricted to the domain $\Omega$ with Dirichlet boundary condition. We then study the equation 
\begin{equation}
\begin{split}
\label{eq:HF2}
i \partial_t u(x,t) &= (H_0 + V_{\operatorname{TD}}(t))u(x,t)+\int_{\Omega} \frac{\vert u(y,t) \vert^2 }{\vert x-y \vert} \ dy \ u(x,t), \quad (x,t) \in \Omega \times (0,T) \\
u(x,t)&=0 \quad (x,t) \in \partial \Omega \times (0,T)\text{ such that }
u(\bullet,0)=u_0 \in H^2(\Omega) \cap H^1_0(\Omega).
\end{split}
\end{equation}
It is easy to see that Theorem \ref{theo:HF} holds true with the same constant $C_{T,\alpha,\rho}$ for the Dirichlet problem.

\subsection{Reduction to bounded domains}
Now we need to show that the solution on the bounded domain tends in $L^\infty(0,T;L^2(\RR^3))$ to the solution on $\RR^3$ as the radius of the ball tends to $\infty$. So first we prove the equivalent of \cite[Lemma 7.1]{BH} for our version of the Schr\"odinger equation.

\begin{lemm}
\label{lm:PertubationOfSingularPotentialsAndInitialStates}
(Perturbation of singular potentials \& initial states). Let $\widetilde{W}_{\operatorname{sing}},W_{\operatorname{sing}},W_{\operatorname{reg}},V_{\operatorname{con}},u$ satisfy the assumptions in Assumption \ref{ass:ass1}, and $\widetilde{\psi_0},\psi_0\in H_2^2(\RR^3)$ be two initial states. Then the solution to
\begin{equation*}
    \begin{split}
        i \partial_t \widetilde{\psi}(x,t) &= (-\Delta + \widetilde{W}_{\operatorname{sing}} + \widetilde{V}_{\operatorname{TD}}(t))\widetilde{\psi}(x,t)+\left(\vert \widetilde{\psi}(\bullet,t) \vert^2*\frac{1}{\vert \bullet \vert} \right)(x)\widetilde{\psi}(x,t)\\
        \widetilde{\psi}(\bullet,0)&=\widetilde{\psi_0}.
    \end{split}
\end{equation*}

converges in $L^\infty (0,T;L^2(\RR^3))$ to the solution of
\begin{equation*}
    \begin{split}
        i \partial_t {\psi}(x,t) &= (-\Delta + {W_{\operatorname{sing}}} + V_{\operatorname{TD}}(t)){\psi}(x,t)+\left(\vert {\psi}(\bullet,t) \vert^2*\frac{1}{\vert \bullet \vert} \right)(x){\psi}(x,t)\\
        {\psi}(\bullet,0)&=\psi_0
    \end{split}
\end{equation*}

as $\widetilde{W}_{\operatorname{sing}} \to_{2} W_{\operatorname{sing}}$, $\widetilde{V}_{\operatorname{con}}\to_{2} V_{\operatorname{con}}$, and $\widetilde{\psi_0}\to_{2}\psi_0$. Furthermore there is
\begin{equation*}
    \begin{split}
        C = C(T,||\psi||_{L^\infty((0,T),H^2(\RR^d))},||\widetilde{\psi}||_{L^\infty((0,T),H^2(\RR^d))})
    \end{split}
\end{equation*}
such that
\begin{equation*}
    \begin{split}
        ||\psi-\widetilde{\psi}||_{L^\infty (0,T;L^2(\RR^3))} \le C\left(||W_{\operatorname{sing}}-\widetilde{W}_{\operatorname{sing}}||_{L^2(\RR^3)}+||V_{\operatorname{con}}-\widetilde{V}_{\operatorname{con}}||_{2}+||\psi_0-\widetilde{\psi_0}||_{L^2(\RR^3)}\right).
    \end{split}
\end{equation*}
\end{lemm}

\begin{proof}
 Let $\xi = \psi-\widetilde{\psi}$ and subtract the two equations to get
\begin{equation*}
    \begin{split}
        i\partial_t \xi = \left(-\Delta +W_{\operatorname{reg}} + \widetilde{W}_{\operatorname{sing}} +V_{\operatorname{TD}}(t) \right)\xi+(W_{\operatorname{sing}}-\widetilde{W}_{\operatorname{sing}})\psi+(F(\psi)-F(\widetilde{\psi}))
    \end{split}
\end{equation*}

with initial condition $\xi(0) = \psi_0-\widetilde{\psi_0}$. Multiply by $\overline{\xi}$, integrate over $\RR^3$ and take the imaginary part to get
\begin{equation*}
    \begin{split}
        \frac{1}{2}\frac{d}{dt}\int_{\RR^3}|\xi|^2 &= \Im\int_{\RR^3}\left((W_{\operatorname{sing}}-\widetilde{W}_{\operatorname{sing}})\psi+F(\psi)-F(\widetilde{\psi})\right)\overline{\xi} \\
        &\leq ||\xi||_{L^2(\RR^3)}\left(\left\lVert (W_{\operatorname{sing}}-\widetilde{W}_{\operatorname{sing}})\psi\right\rVert_{L^2(\RR^3)} + \left\lVert F(\psi)-F(\widetilde{\psi})\right\rVert_{L^2(\RR^3)}\right).
    \end{split}
\end{equation*}

Now note that
\begin{equation*}
    \begin{split}
        \left\lVert (W_{\operatorname{sing}}-\widetilde{W}_{\operatorname{sing}})\psi \right\rVert_{2} \lVert\xi\rVert_{2}  &\leq \left\lVert W_{\operatorname{sing}}-\widetilde{W}_{\operatorname{sing}} \right\rVert_{2}^2 \lVert \psi \rVert ^2_{H^2}/2+ \lVert\xi\rVert_{2}^2/2.
    \end{split}
\end{equation*}

Integrating in time and using that $\psi\in L^\infty((0,T),H^2(\RR^3))$, Lemma \ref{EuanRulesForF} yields
\begin{equation*}
    \begin{split}
        ||\xi(t)||^2_{L^2(\RR^3)}
        &\leq ||\xi(0)||^2_{L^2(\RR^3)}+\left\lVert W_{\operatorname{sing}}-\widetilde{W}_{\operatorname{sing}}\right\rVert^2_{2}\int_0^t ||\psi(s)||^2_{H^2}ds+\int_0^t C||\xi(s)||_{2}^2ds\\
        &\leq  ||\xi(0)||^2_{L^2(\RR^3)}+Ct\left\lVert W_{\operatorname{sing}}-\widetilde{W}_{\operatorname{sing}}\right\rVert^2_{2}+\int_0^t C||\xi(s)||_{2}^2ds
    \end{split}
\end{equation*}

where the constants depend on $||\psi||_{L^\infty((0,T),H^2(\RR^d))},||\widetilde{\psi}||_{L^\infty((0,T),H^2)} $.

Now Gr\"onwall's inequality implies that
\begin{equation*}
    \begin{split}
        ||\xi||_{L^\infty((0,T),L^2(\RR^3))}\lesssim  \left\lVert W_{\operatorname{sing}}-\widetilde{W}_{\operatorname{sing}}\right\rVert^2_{2}+||\psi_0-\widetilde{\psi_0}||_{L^2(\RR^3)}.
    \end{split}
\end{equation*}
The control potential can be included in a similar way.
\end{proof}

Since $u \in W^{1,1}_{\operatorname{pcw}}$ it suffices to treat from now on the case that $u$ is constant, since $u$ can be approximated by piecewise constant functions.

Now, we are ready to prove the reduction to a boundary value problem on a bounded domain.

\begin{theo}
\label{theo:auxBVP}
Let $\psi_0\in H^2_2(B_R(0))\cap H^1_0(B_R(0))$ be an initial state, $T\in (0,\infty)$ and  $V,V_{\operatorname{con}},u$ as in Assumption \ref{ass:ass1}. Then the equation with solution $\psi$ on $\mathbb R^3$ can be approximated by $\psi^{D}$, the solution to the Dirichlet boundary value problem on $B_R(0)$ for $R>0$. In particular, the difference $\xi = \psi-\psi^{D}$ satisfies, for some constant 
\[C = C(T, ||u||_{W^{1,1}_{\operatorname{pcw}}(0,T)},||W_{\operatorname{sing}}||_{2},||\langle\bullet\rangle^{-2}W_{\operatorname{reg}}||_{L^\infty},||\langle\bullet\rangle^{-2}V_{\operatorname{con}}||_{L^\infty})\]
we have
\begin{equation*}
    \begin{split}
        ||\xi||^2_{L^\infty((0,T),L^2(\RR^3))}\le C \left|\sup_{t\in(0,T)} \int_{\partial B_R(0)}\xi\nabla \bar \xi \ dS \right|  \lesssim \frac{||\psi_0||_{H^2_2(\RR^3)}}{R^2}.
    \end{split}
\end{equation*}
\end{theo}
\begin{proof}
First we use Lemma \ref{lm:PertubationOfSingularPotentialsAndInitialStates} to reduce the initial state to an initial state with support in $B_R(0)$. So we separate the solution $\psi$ into its behaviour outside the ball $B_R(0)$, where the solution to the boundary value problem is zero, and it's behaviour inside $B_R(0)$.

To control the difference inside the ball, take the difference $\xi=\psi-\psi^{D}$ such that 
\begin{equation*}
    \begin{split}
        i\partial_t\xi(t) &= (-\Delta+V+V_{\operatorname{TD}}(t))\xi(t)+F(\psi(t))-F(\psi^{D}(t))\text{ with }\xi(0) = 0.
    \end{split}
\end{equation*}

Multiplying by $\overline{\xi}$, integrate over $B_R(0)$, and taking the imaginary part yields
\begin{equation*}
    \begin{split}
        \frac{1}{2}\frac{d}{dt}&\int_{B_R(0)}|\xi(x,t)|^2dx = \Im\left(\int_{B_R(0)}-\overline{\xi}\Delta\xi + (F(\psi)-F(\psi^{D}))\overline{\xi} \ dx\right)\\
        &\leq \left| \int_{\partial B_R(0)}\overline{\xi}\nabla\xi\ dS \right|+\int_{B_R(0)}|F(\psi)-F(\psi^{D})||\overline{\xi}| \ dx\\
        &\lesssim \left| \int_{\partial B_R(0)}\overline{\xi}\nabla\xi\ dS \right|+ ||\xi||_{L^2(B_R(0))}^2.
    \end{split}
\end{equation*}

Gr\"onwall's inequality implies then that $||\xi(t)||^2_{L^2(B_R(0))}\lesssim C\left| \int_{\partial B_R(0)}\overline{\xi}\nabla\xi \ dS \right|.$ Combining the two bounds above with Lemma \ref{lemm:decay} yields then the bound of this Lemma.
\end{proof}
\subsection{Approximation of Coulomb kernel}

Next we aim to replace the Coulomb kernel, $1/|x|,$ with $f\in C^\infty(\Omega) \cap L^2(\Omega)$ as we need smooth functions in order to compute the solution. 
It is easy to see that Theorem \ref{theo:HF} holds true for $1/|x|$ replaced by $f.$
\begin{lemm}
Let $\psi$ be the solution to the Hartree equation and $\widetilde{\psi}$ be the solution with $1/\vert x \vert$ replaced by $f \in C^{\infty}(\Omega) \cap L^2(\Omega)$.

The difference of the two solutions $\xi_:=\widetilde{\psi}-\psi$ then satisfies the PDE
\begin{equation}
    \begin{split}
    \label{eq:difference}
        i\partial_t\xi=(H_0+V_{\operatorname{TD}})\xi+\int_\Omega f(\bullet-y)|\psi_n(y)|^2dy \ \widetilde{\psi}-\int_\Omega\frac{1}{|x-y|}|\psi(y)|^2dy \ \psi
    \end{split}
\end{equation}
and has the property that $ \lVert\xi\rVert^2_{L^\infty(0,T;L^2)}\leq C_T\left\lVert f-\frac{1}{|\bullet|}\right\rVert_{1}$.

\end{lemm}
\begin{proof}
Now multiply \eqref{eq:difference} by $\overline{\xi}$, take the imaginary part and integrate over $\Omega$ to get
\begin{equation*}
    \begin{split}
        \frac{1}{2}\frac{d}{dt}\Vert \xi \Vert^2_{2}
        &= \Im\int_\Omega\int_\Omega \left(f(x-y)|\widetilde{\psi}(y)|^2-\frac{1}{|x-y|}|\psi(y)|^2\right)dy \ \psi(x)\overline{\xi}(x)dx\\
        &\leq \left\lVert\int_\Omega f(\bullet-y)|\widetilde{\psi}(y)|^2-\frac{1}{|\bullet-y|}|\psi(y)|^2dy\right\rVert_{2}\Vert\psi\Vert_{L^\infty}\Vert\xi\Vert_{2}\\
        &=\left\lVert\int_\Omega \left(f(\bullet-y)-\frac{1}{|\bullet-y|}\right)|\widetilde{\psi}(y)|^2+\frac{1}{|\bullet-y|}\left(|\widetilde{\psi}(y)|^2-|\psi(y)|^2\right)dy\right\rVert_{2}\Vert\psi\Vert_{L^\infty}\Vert\xi\Vert_{2}.
    \end{split}
\end{equation*}

We can bound the first term in the preceding equation by Young's convolution inequality
\begin{equation*}
    \begin{split}
       &\left\lVert\ \left(f-\frac{1}{|\bullet|}\right)*|\widetilde{\psi}(\bullet)|^2\right\rVert_{2}\leq\left\lVert f-\frac{1}{|\bullet|}\right\rVert_{1}\lVert|\widetilde{\psi}|^2\rVert_{2}\lesssim\left\lVert f-\frac{1}{|\bullet|}\right\rVert_{1}\lVert\widetilde{\psi}\rVert_{H^2_2}^2
    \end{split}
\end{equation*}
and similarly for the second term
\begin{equation*}
    \begin{split}
        \left\lVert \frac{1}{|\bullet|}*\left(|\widetilde{\psi}(\bullet)|^2-|\psi(\bullet)|^2\right) \right\rVert_{2} &\leq \left\lVert \frac{1}{|\bullet|}\right\rVert_{2} \left\lVert |\widetilde{\psi}(\bullet)|^2-|\psi(\bullet)|^2 \right\rVert_{1}\\
        &= C \int_\Omega |\psi(x)+\xi(x)|^2-|\psi(x)|^2 \ dx\\
        &\lesssim\int_\Omega |\psi(x)||\xi(x)|+|\widetilde{\psi}(x)||\xi(x)|\ dx\\
        &\lesssim\left(\lVert\psi\rVert_{2}+\lVert\widetilde{\psi}\rVert_{2}\right)\lVert\xi\rVert_{2}.
    \end{split}
\end{equation*}

Now using these bounds and that $\lVert\psi\rVert_{H^2_2},\lVert\widetilde{\psi}\rVert_{H^2_2}\lesssim\lVert\psi_0\rVert_{H^2_2}$ 
\begin{equation*}
    \begin{split}
        \frac{d}{dt}\lVert\xi\rVert^2_{2} &\lesssim\left(\left\lVert f-\frac{1}{|\bullet|}\right\rVert_{1}\lVert\psi_0\rVert_{H^2_2}^2+\lVert\psi_0\rVert_{H^2_2}\lVert\xi\rVert_{2}\right)\lVert\psi_0\rVert_{H^2_2}\lVert\xi\rVert_{2}\\
        &\lesssim\left\lVert f-\frac{1}{|\bullet|}\right\rVert_{1}\lVert\psi_0\rVert_{H^2_2}^3+C\lVert\psi_0\rVert_{H^2_2}^2\lVert\xi\rVert_{2}^2.
    \end{split}
\end{equation*}

Integrating with respect to time gives and an application of Gr\"onwall's inequality yields
\begin{equation*}
    \begin{split}
        \lVert\xi(t)\rVert_{2}^2\lesssim\left\lVert f-\frac{1}{|\bullet|}\right\rVert_{1}t\exp{Ct}.
    \end{split}
\end{equation*}

Hence, $ \lVert\xi\rVert^2_{L^\infty(0,T;L^2)}\leq C_T\left\lVert f-\frac{1}{|\bullet|}\right\rVert_{1}$.
\end{proof}

\begin{ex}
We can therefore choose the smooth function $f_{\varepsilon}(x):=\frac{1}{(\vert x \vert^2+\varepsilon^2)^{1/2}}$ which implies that on any bounded domain $\Omega$ 
\[\Vert f_{\varepsilon}-1/|\bullet | \Vert_{L^1} = \mathcal O_{\Omega}(\varepsilon).\]

\end{ex}
\section{Numerical methods}

We have shown so far that we can replace both the Hartree equation, in Theorem \ref{theo:auxBVP}, and linear magnetic Schrödinger equation, in Theorem \ref{theo:Dirichlet}, with the same equation but now on a bounded domain. 
\label{sec:numerics}
\subsection{Continuous Strang splitting scheme} 
For our numerical approximation, we consider the following continuous \emph{Strang splitting scheme} of the solution on the bounded domain $\Omega:$
We now want to approximate the solution $\psi(t_n)$ with $t_n=n\tau$ for a step size $\tau>0$ by functions $\psi_n$
\begin{equation}
    \begin{split}
    \label{eq:strangsplitting}
 \psi_{n+1/2}^{-}&=e^{-\frac{i}{2} \tau (-i\nabla-A)^2 }\psi_n, \\
 \psi_{n+1/2}^+ &= e^{-i\tau \mathcal V[\psi_{n+1/2}^-] }\psi_{n+1/2}^- ,\text{ and }\\
 \psi_{n+1} &= e^{-\frac{i}{2} \tau (-i\nabla-A)^2}\psi_{n+1/2}^+.
 \end{split}
\end{equation}
Here, $\mathcal V[\psi]:=f*\vert \psi \vert^2+V$ where $V$ includes all contributions to the potential and $e^{-\frac{i}{2} \tau (-i\nabla-A)^2}\psi$ is the solution to the linear Schr\"odinger equation 
\begin{equation*}
\begin{split}
i \partial_t u(x,t) &= (-i\nabla-A)^2 u(x,t), \quad (x,t) \in \Omega \times (0,T) \\
u(x,t)&=0 \quad (x,t) \in \partial \Omega \times (0,T)\\
u(\bullet,0)&=\psi \in H^2(\Omega) \cap H^1_0(\Omega).
\end{split}
\end{equation*}
evaluated at time $t=\tau/2.$ 

\subsubsection{Perturbation of the potential}

\begin{lemm}
\label{lemm:potpert}
Let $V,U \in L^2(\Omega)$ be potentials and denote by $\psi_V,\psi_U,$ two solutions to the Hartree equation \eqref{eq:HF2} with the respective time-independent potential with Dirichlet boundary conditions and some initial state $\varphi_0 \in H^2(\Omega)\cap H^1_0(\Omega)$, then for some constant $C(T)>0$
\[ \Vert \psi_V-\psi_{U} \Vert_{2} \le C(T, \varphi_0) \Vert V-U \Vert_{2}.\]
\end{lemm}
\begin{proof}
Looking at the difference solution $\xi=\psi_V-\psi_{U}$, we find using Lemma \ref{EuanRulesForF}
\begin{equation*}
\begin{split}
\frac{1}{2} \frac{d}{dt} \int_{\Omega} \vert \xi(x,t) \vert^2 \ dx &= \Im\left( \int_{\Omega} ((V-U)\psi_V \bar \xi)(x) + F(\psi_U(x))-F(\psi_V(x)) \ dx \right)\\
&\le \Vert V-U \Vert_{2} \Vert \psi_V \Vert_{\infty} \Vert  \xi \Vert_2+C_F(\Vert \psi_U \Vert_{H^2_2}+\Vert \psi_V \Vert_{H^2_2})  \Vert  \xi \Vert^2_2 \\
&\le \Vert V-U \Vert_{2}^2 \Vert \psi_V \Vert_{\infty}^2 +(1+ C_F(\Vert \psi_U \Vert_{H^2_2}+\Vert \psi_V \Vert_{H^2_2}) )\Vert  \xi \Vert_2^2
\end{split}
\end{equation*}
which after a straightforward application of Gronwalls lemma implies the claim.
\end{proof}
\begin{rem}[Control functions]
A very similar argument, as in the proof of Lemma \ref{lemm:potpert} shows that the solution to the magnetic Schrödinger or Hartree equation is Lipschitz continuous with respect to the control function $u$ in the $L^1$-norm. We may therefore assume without loss of generality that the control function is piecewise constant in time and neglect the explicit time-dependence of the control function in our discretization scheme.
\end{rem}

\subsubsection{Bounding the convolution}

To estimate the convolution in the Hartree nonlinearity, we need the following auxiliary lemma whose proof follows from elementary estimates that we shall leave to the reader: \begin{lemm}
\label{lemm:L2BoundOfFStar}
Let $\Omega$ be a domain. For $u,v,w\in L^2(\Omega)$, and $f\in L^{\infty}(\Omega)$
\begin{equation*}
    \begin{split}
        \lVert (f*(uv))w\rVert_{2}\leq \Vert f \Vert_{\infty}\Vert u\Vert_{2}\Vert v\Vert_{2}\Vert w\Vert_{2}.
    \end{split}
\end{equation*}
For $u,v,w\in H^1(\Omega)$ we have
\begin{equation*}\begin{split}
    \lVert (f*(uv))w\rVert_{H^1}\lesssim\Vert f \Vert_{\infty}(\Vert u\Vert_{H^1}\Vert v\Vert_{2}\Vert w\Vert_{2}+\Vert u\Vert_{2}\Vert v\Vert_{H^1}\Vert w\Vert_{2}+\Vert u\Vert_{2}\Vert v\Vert_{2}\Vert w\Vert_{H^1})
\end{split}\end{equation*}

and for $u,v,w\in H^2(\Omega)$ 
\begin{equation*}\begin{split}
    \lVert(f*(uv))w\rVert_{H^2}\lesssim \Vert f \Vert_{\infty}\sum_{\substack{k,l,m \in \mathbb N_0\\ k+l+m=2}}\Vert u\Vert_{H^k}\Vert v\Vert_{H^l}\Vert w \Vert_{H^m}.
\end{split}\end{equation*}
\end{lemm}

Our global convergence proof now follows from combining a stability estimate in Subsection \ref{sec:stability}, together with error propagation estimate in Subsection \ref{sec:Localerror}. This will lead us to \eqref{eq:halfway} which shows the convergence of full solution to a time-discretized Splitting scheme. We then study in Subsection \ref{sec:Spacedisc} an additional space-discretization to obtain a complete numerical discretization scheme. The final result is then summarized in Theorem \ref{theo:numerics}.

\subsubsection{Stability of the splitting scheme}
\label{sec:stability}
For the time-evolution in the following Lemma, we introduce the magnetic Sobolev norm
$$ \Vert \psi \Vert_{H^2_B}:=\Vert \psi \Vert_{L^2} + \Vert (-i\nabla-A)^2 \psi \Vert_{L^2}$$
which is clearly equivalent to the usual $H^2$ norm on bounded domains for magnetic vector potentials as specified in Assumption \ref{ass:ass1}.

\begin{lemm}[$L^2$ and $H^2$ stability] We consider the time-discrete strang splitting scheme on some bounded domain $\Omega$ for initial states $\psi_0 \in H^2(\Omega) \cap H^1_0(\Omega)$ with a magnetic field as in Assumption \ref{ass:ass1}
\begin{equation*}
    \begin{split}
        \psi_{n+1/2}^- &= e^{-\frac{i}{2}\tau (-i\nabla-A)^2}\psi_n\\
        \psi_{n+1/2}^+ &= e^{-i\tau\left( \left(f*|\psi_{n+1/2}^-|^2\right) + V\right)}\psi_{n+1/2}^-\\
        \psi_{n+1}&=e^{-\frac{i}{2}\tau(-i\nabla-A)^2}\psi_{n+1/2}^+.
    \end{split}
\end{equation*}
We then introduce the time-evolution operator by $\Phi_{\tau} \psi_n:=\psi_{n+1}.$

Let $\psi,\phi\in H^2$ be two states with $\Vert\psi\Vert_{H^2},\Vert\phi\Vert_{H^2}\leq M_1$, then
\begin{equation*}
    \begin{split}
        \Vert \Phi_\tau(\psi)-\Phi_\tau(\phi)\Vert_{2}&\leq e^{c_0\tau}\Vert\psi-\phi\Vert_{2}\text{ and }
        \Vert \Phi_\tau(\psi)-\Phi_\tau(\phi)\Vert_{H^2_B}\leq e^{c_1\tau}\Vert\psi-\phi\Vert_{H^2_B}.
    \end{split}
\end{equation*}

\end{lemm}

\begin{proof}
 As $e^{-\frac{i}{2}\tau (-i\nabla-A)^2}$ preserves the $L^2$ and $H^2_B$ norm, we only need to consider the evolution by $e^{-i\tau\left( \left(f*|\psi|^2\right) + V\right)}\psi.$
 We shall continue with the $L^2$ bound and observe that $e^{-i\tau\left( \left(f*|\psi|^2\right) + V\right)}\psi=\theta(\tau)$ and $e^{-i\tau\left( \left(f*|\phi|^2\right) + V\right)}\phi=\eta(\tau)$ where $\theta$ and $\eta$ are solutions to
 \begin{equation*}
     \begin{split}
         i\Dot{\theta} &= (f*|\psi|^2+V)\theta, \quad \theta(0)=\psi\\
         i\Dot{\eta}&=(f*|\psi|^2+V)\eta,\quad\eta(0)=\phi.
     \end{split}
 \end{equation*}

So taking the difference
\begin{equation*}
    \begin{split}
        i(\Dot{\theta}-\Dot{\eta})&=(f*|\psi|^2+V)\theta-(f*|\phi|^2+V)\eta\\
        &= (f*(\psi-\phi)\overline{\psi})\theta+(f*(\overline{\psi}-\overline{\phi})\phi)\theta+(f*|\phi|^2)(\theta-\eta)+V(\theta-\eta).
    \end{split}
\end{equation*}

Integrating this identity in time, we find by the unit $L^2$ norm of $\phi$ and $\psi$, and the initial state of $\theta(0)=\psi$, $\eta(0)=\phi$ that
\begin{equation*}
    \begin{split}
        \lVert\theta(t)-\eta(t)\rVert_{2}\leq \lVert \psi-\phi\rVert_{2}+ K_0\lVert\psi-\phi\rVert_{2}(\lVert \psi\rVert_{H^1}+\lVert \phi\rVert_{H^1})t\\\quad+\left(K_0\rVert\phi\rVert_{H^1}+\lVert V\rVert_{W^{1,\infty}}\right)\int_0^t\lVert\theta(s)-\eta(s)\rVert_{2} \ ds.
    \end{split}
\end{equation*}

By Gr\"onwall's inequality
\begin{equation*}
    \begin{split}
        \lVert\theta(t)-\eta(t)\rVert_{2}\leq\Vert \psi-\phi\Vert_{2}(1+K_0(\Vert\psi\Vert_{H^1}+\Vert\phi\Vert_{H^1})t)\exp\left(\left(K_0\Vert\phi\Vert_{H^1}+\lVert V\rVert_{W^{1,\infty}}\right)t\right)
    \end{split}
\end{equation*}
and thus, writing $\mathcal V[\psi]:=f*\vert \psi \vert^2+V$,
\begin{equation*}
    \begin{split}
        \Vert e^{-i\tau \mathcal V[\psi]}\psi-e^{-i\tau \mathcal V[\phi]}\phi\Vert_{2} &= \Vert \theta(\tau)-\eta(\tau)\Vert_{2}\\
        &\leq \lVert \psi-\phi\rVert_{2}(1+2K_0M_1\tau)\exp((K_0M_1+\lVert V \rVert_{W^{1,\infty}})\tau)\\
        &\le\lVert \psi-\phi\rVert_{2}e^{c_0\tau}.
    \end{split}
\end{equation*}

The proof of the $H^2_B$ estimate follows in a similar fashion.
\end{proof}

\subsubsection{Error propagation}
\label{sec:Localerror}
Using a Lie algebraic description, that is largely inspired by the analysis in \cite{L}, we want to bound the error in one step in the splitting scheme. First we let $\hat T,\hat V$ be vector fields defined by
\begin{equation}
\begin{split}
\hat T(\psi) &= -i(-i\nabla-A)^2\psi=i\Delta \psi +i\langle\nabla,A\psi\rangle+i\langle A, \nabla \psi\rangle+i \vert A \vert^2 \psi,\text{ and }\\
\hat{V}(\psi) &= -i(f*|\psi|^2+V)\psi,
\end{split}
\end{equation}
then their Lie commutator is
\begin{equation*}
    \begin{split}
        [\hat T,\hat V](\psi) =& \hat T'\vert_{\psi}(\hat V(\psi))-\hat V'\vert_{\psi}(\hat T(\psi))\\
        =& i\Delta(-i(f*\psi\overline{\psi}-V)\psi)+i(f*\psi\overline{\psi}-V)i\Delta\psi + i(f*i\Delta\psi\overline{\psi})\psi+i(f*\psi\overline{i\Delta\psi})\psi\\
        &+ i\langle \nabla, -i A(f*\psi\overline{\psi}-V) \rangle \psi +i \langle A, -i \nabla(f*\psi\overline{\psi}-V) \rangle \psi+2(f*\vert A\vert^2\vert \psi \vert^2) \psi \\ 
        &+ if*(\langle \nabla, -i A\psi \rangle\overline{\psi}+\overline{\langle \nabla, -i A\psi \rangle}\psi )\psi +i f*(\langle A, -i \nabla\psi \rangle\overline{\psi}+\overline{\langle A, -i \nabla\psi \rangle}\psi) \psi\\
        =& (\langle 2f*\nabla\psi,\nabla\overline{\psi} \rangle+2f*\psi\Delta\overline{\psi}+\Delta V)\psi + 2\langle (f*\nabla\psi \ \overline{\psi}+f*\psi \ \nabla\overline{\psi}+\nabla V),\nabla\psi \rangle\\
        &+ i\langle \nabla, -i A(f*\psi\overline{\psi}-V) \rangle \psi +i \langle A, -i \nabla(f*\psi\overline{\psi}-V) \rangle \psi+2(f*\vert A\vert^2\vert \psi \vert^2) \psi \\
        &+ if*(\langle \nabla, -i A\psi \rangle\overline{\psi}+\overline{\langle \nabla, -i A\psi \rangle}\psi )\psi +i f*(\langle A, -i \nabla\psi \rangle\overline{\psi}+\overline{\langle A, -i \nabla\psi \rangle}\psi) \psi.
    \end{split}
\end{equation*}
We thus have for the $L^2$ norm of the commutator that
\begin{equation}
    \begin{split}
    \label{eq:commutator}
        \lVert [\hat{T},\hat{V}](\psi) \rVert_{2} &\leq C\lVert \psi\rVert_{H^2_2}\lVert\psi\rVert_{H^1_1}\lVert\psi\rVert_{2}+C\lVert V \rVert_{W^{2,\infty}}\lVert\psi\rVert_{H^1_1}.
    \end{split}
\end{equation}
In the following, we write $\psi(\tau)=:\operatorname{exp}(\tau D_K)\psi_0$ to denote the solution to the equation 
\[ \psi'(t) = K(\psi(t)), \quad \psi(0)=\psi_0.\]
By applying the variation of constants formula twice, we find for the solution
\begin{equation*}
    \begin{split}
        \psi(\tau)&=\exp(\tau D_H)\operatorname{id}(\psi_0)\\
        &=\exp(\tau D_{\hat T})\operatorname{id}(\psi_0)+\int_0^\tau \exp((\tau-s)D_H)D_{\hat V}\exp(sD_{\hat T})\operatorname{id}(\psi_0) \ ds\\
        &=\exp(\tau D_{\hat T})\operatorname{id}(\psi_0)+\int_0^\tau\exp ((\tau-s)D_{\hat T})D_{\hat V}\exp(sD_{\hat T})\operatorname{id}(\psi_0) \ ds+r_1(\psi_0),
    \end{split}
\end{equation*}
with remainder term
\begin{equation*}
    \begin{split}
        r_1(\psi_0)=\int_0^\tau\int_0^{\tau-s}\exp((\tau-s-\sigma)D_H)D_{\hat V}\exp(\sigma D_{\hat T})D_{\hat V}\exp(sD_{\hat T})\operatorname{id}(\psi_0) \ d\sigma \ ds.
    \end{split}
\end{equation*}

We can also write the time-discretization of the Strang splitting \eqref{eq:strangsplitting} on continuous space $\Omega$ as
\begin{equation*}
    \begin{split}
        \psi_1=\exp(\tfrac{\tau}{2} D_{\hat T})\exp(\tau D_{\hat V})\exp(\tfrac{\tau}{2} D_{\hat T})\operatorname{id}(\psi_0).
    \end{split}
\end{equation*}

By Taylor expansion, we have $\exp(\tau D_{\hat V})=I+\tau D_{\hat V}+\tau^2\int_0^1(1-\theta)\exp(\theta\tau D_{\hat V})D_{\hat V}^2d\theta$
\begin{equation*}
    \begin{split}
        \psi_1=\exp(\tau D_{\hat T})\operatorname{id}(\psi_0)+\tau \exp(\tfrac{\tau}{2} D_{\hat T})D_{\hat V}\exp(\tfrac{\tau}{2} D_{\hat T})\operatorname{id}(\psi_0)+r_2(\psi_0)
    \end{split}
\end{equation*}

with remainder $r_2(\psi_0)=\tau^2\int_0^1(1-\theta)\exp(\frac{\tau}{2} D_{\hat T})\exp(\theta\tau D_{\hat V})D_{\hat V}^2\exp(\frac{\tau}{2}D_{\hat T})\operatorname{id}(\psi_0)d\theta.$

Therefore, the error between the actual solution and the time-discretization becomes
\begin{equation*}
    \begin{split}
        \psi_1-\psi(\tau)&=\tau\exp(\tfrac{\tau}{2} D_{\hat T})D_{\hat V}\exp(\tfrac{\tau}{2} D_{\hat T})\operatorname{id}(\psi_0)\\
        &-\int_0^\tau \exp((\tau-s)D_{\hat T})D_{\hat V}\exp(sD_{\hat T})\operatorname{id}(\psi_0)ds+(r_2-r_1)(\psi_0).
    \end{split}
\end{equation*}

Let $f(s) := \exp((\tau-s)D_{\hat T})D_{\hat V}\exp(sD_{\hat T})\operatorname{id}(\psi_0)$, then we have 
\begin{equation}
\label{eq:difference2}
\psi_1-\psi(\tau) = \int_0^\tau f(\tau/2)-f(s) \ ds+(r_2-r_1)(\psi_0).
\end{equation}

Thus, apart from the two remainder expressions, we are left to understand the operator associated with the mid-point rule $L(g)=\tau g(\frac{1}{2}\tau)-\int_0^\tau g(s)ds$. Using Peano's kernel theorem with the Peano kernel 
\begin{equation*}
    \begin{split}
        k_1(\theta) &= L[(\bullet-\theta)\indic_{\bullet\geq\theta}] = \tau\left(\theta-\frac{\tau}{2}\right)\indic_{\tau/2\leq\theta}-\theta^2/2.
    \end{split}
\end{equation*}
we are then left to study an operator, which by integration by parts, since $k_1(0)=k_1(\tau)=0$, satisfies 
\begin{equation*}
\begin{split}
    Lf &= \int_0^\tau k_1(\theta)f^{(2)}(\theta)d\theta=- \int_0^\tau k_1'(\theta)f'(\theta) d\theta.
    \end{split}
\end{equation*}
Now, clearly $|k_1'(\theta)|=|\tau\indic_{\tau/2\leq\theta}-\theta|\leq\tau$ and
\begin{equation*}
    \begin{split}
        f'(s) &= -\exp((\tau-s)D_{\hat T})[D_{\hat T},D_{\hat V}]\exp(sD_{\hat T})\operatorname{id}(\psi_0)\\
        &=\exp((\tau-s)D_{\hat T})D_{[\hat T,\hat V]}\exp(sD_{\hat T})\operatorname{id}(\psi_0)\\
        &=e^{-is(-i \nabla-A)^2}[\hat T,\hat V](e^{-i(\tau-s)(-i \nabla-A)^2}\psi_0).
    \end{split}
\end{equation*}
which we can bound, using  \eqref{eq:commutator}, as
\[ \lVert f'(s)\rVert_{2}\leq C(\lVert\psi\rVert_{H^2_2}\lVert\psi\rVert_{H^1_1}\lVert\psi\rVert_{2}+\lVert V\rVert_{W^{2,\infty}}\lVert\psi\rVert_{H^1_1}).\]

This implies that 
\begin{equation}
\label{eq:Lf}
\lVert Lf \rVert_{2}\leq C\tau^2(\lVert\psi\rVert_{H^2_2}\lVert\psi\rVert_{H^1_1}\lVert\psi\rVert_{2}+\lVert V\rVert_{W^{2,\infty}}\lVert\psi\rVert_{H^1_1}).
\end{equation}

So now estimating the remainder terms in \eqref{eq:difference2} using
\begin{equation*}
    \begin{split}
        \lVert \hat{V}(\psi)\rVert_{2} &\leq C\Vert f \Vert_{\infty}\lVert\psi\rVert_{2}^3+\lVert V\rVert_{W^{\infty}}\lVert \psi\rVert_{2}\\
        \lVert \hat{V}'(\psi)\phi\rVert_{2} &\leq C\Vert f \Vert_{\infty}\lVert\psi\rVert_{2}^2\Vert \phi \Vert_2+\lVert V\rVert_{W^{\infty}}\lVert \phi\rVert_{2},
    \end{split}
\end{equation*}
 we infer that $\lVert r_1\rVert_{2}+\lVert r_2\rVert_{2}\leq C_2\tau^2$ and hence by combining the remainder estimates with \eqref{eq:Lf} $\lVert\psi_1-\psi(\tau)\rVert_{2}\leq C_4\tau^2.$

By the usual Lady Windermere's fan argument, this implies that
\begin{equation}
\label{eq:halfway}
 \lVert\psi_N-\psi(N\tau)\rVert_{2}\leq C(T)\tau.
 \end{equation}

This shows the explicit global convergence of our time-discretized but space-continuous splitting scheme. Our next goal is to obtain a space-discretized version of \eqref{eq:halfway}.

\subsection{Space-Discrete Splitting scheme}
\label{sec:Spacedisc}
To define a convergent numerical discretization scheme, we employ a \emph{cubic discretization}.

\textbf{Cubic Discretization:} 
Consider a lattice of side length $h$ with lattice points $(x_j)_{j\in\ZZ^3}$ and a family of cubes $Q_{x_j}=\times_{i=1}^3[x_j^i-\tfrac{h}{2},x_j^i+\tfrac{h}{2})$ with $j\in\ZZ^3$ that form a disjoint decomposition of $\RR^3$ up to a set of measure zero. The cubic approximation of a function $f\in L_{loc}^1(\RR^3,\mathbb{C})$ is defined by
\begin{equation*}
    \begin{split}
        f_Q(x):=\sum_{j\in\ZZ^3} \frac{1}{\vol(Q_{x_j})}\int_{Q_{x_j}}f(s) \ ds \indic_{Q_{x_j}}(x). 
    \end{split}
\end{equation*}

Now we need to introduce the discrete derivative, so let $(\tau_h^if)(x)=f(x-h\hat{e_i})$ be the translation by $h$ and $\delta^i_h=(\tau^i_{-h}-\tau^i_h)/(2h)$ the discretized symmetric derivative in direction $i$ with step $h>0$. Then we can define the discretized Laplacian $\Delta^h=\sum_{i=1}^3(\delta^i_h)^2$ and discretized gradient $\nabla^h=(\delta_h^i)_i$.

Then it is easy to verify that $(\delta^i_h)^*=-\delta_h^i$ and also the product rule holds
\[(\delta_hfg)(x)=f(x+h)(\delta_hg)(x)+g(x-h)(\delta_hf)(x).\]
Moreover, by \cite[Prop. 8.9]{BH}, for $n\in\ZZ_+, k\in{1,2,3}, f\in W^{n,p}(\RR^d)$ and $p\in[1,\infty)$ we have
\begin{equation*}
    \begin{split}
        \lVert ((\delta_h^k)^n-\partial_k^n)f\rVert_{L^p(\RR^3)}=o(1) \operatorname{ as } h\to0
    \end{split}
\end{equation*}

and for $f\in W^{n+1,p}(\RR^3)$ and $p\in[1,\infty]$ there is a constant $C>0$ independent of $f$ and $h$ such that
\begin{equation*}
    \begin{split}
        \lVert (\delta_h^k)^nf_Q-\partial^n_kf\rVert_{L^p(\RR^3)}\leq C\lVert f\rVert_{W^{n+1,p}(\RR^3)}h \  \operatorname{  and  } \ \lVert (\delta_h^k)^nf\rVert_{L^p}\leq C\lVert f\rVert_{W^{n,p}(\RR^d)}.
    \end{split}
\end{equation*}
In addition, for $p\in(1,\infty)$, $f\in W^{n+\epsilon,p}(\RR^3)$, and any $\epsilon\in(0,1]$ it follows that 
\[\lVert (\delta_h^k)^nf_Q-\partial_k^nf\rVert_{L^p(\RR^3)}=\mathcal{O}(\lVert f\rVert_{W^{n+\epsilon,p}(\RR^3)}h^\epsilon).\]

We can now directly use \cite[Lemma 8.10]{BH} that for our Schr\"odinger equation using the discretized derivative there is a constant $C=C(T)$ such that $ \lVert \psi^h\rVert_{L^\infty((0,T),H^2(\Omega))}\leq C\lVert \phi_0\rVert_{H^2(\Omega)}.$

Now we need the equivalent of \cite[Lemma 8.11]{BH}, whose proof is a simple adaptation of the proof given in \cite{BH}, but with a $H^2_h$ norm and a magnetic Schrödinger operator, instead.

\begin{lemm}
\label{lemm:lemm1}
For an initial state $\phi_0\in H^{2+\epsilon}(\Omega)$ the difference between $\psi$, the solution to the Schr\"odinger equation $
        i\partial_t\psi=(-i \nabla-A)^2\psi$ with initial state $\phi_0,$
and the solution $\psi^h$, to the discretized Schr\"odinger equation $i\partial_t\psi^h=H^h_{\operatorname{lin}}\psi^h$ with initial state $\psi^h(0)=(\phi_0)_Q$ and Hamiltonian $H^h_{\operatorname{lin}}:=(-i\nabla^h-A^h)^2$
satisfies for some constant $C=C(T,\lVert\phi_0\rVert_{H^{2+\epsilon}})$
\begin{equation*}
    \begin{split}
        \lVert \psi-\psi^h\rVert_{L^\infty((0,T),L^2)}\leq Ch^\epsilon.
    \end{split}
\end{equation*}

\end{lemm}

To analyze the second step of the splitting scheme, where we ought to compare the dynamics induced by  $V[\psi]=f*|\psi|^2+V=\int_\Omega f(x-y)|\psi(y)|^2 \ dy+V$ with the one in which $f$ is replaced by $f_Q$ and $V$ by $V_Q$

\begin{lemm}
\label{lemm:lemm2}
For an initial state $\phi_0\in H^{2}(\Omega)$, potential $V\in W^{3,\infty}(\Omega)$ and $f\in C^\infty(\Omega)$ then there is a constant $C=C(\lVert\phi_0\rVert_{H^{2}},\lVert V\rVert_{2},T)$ such that if $\psi$ and $\psi^h$ are respectively the solutions to
\begin{equation*}
    \begin{split}
        i\partial_t\psi &= (f*|\phi_0|^2+V)\psi\text{ with }
        \psi(0) = \phi_0 \text{ and }\\
        i\partial_t\psi^h &= (f_Q*|\phi_0|^2_Q+V_Q)\psi^h\text{ with }
        \psi^h_0=(\phi_0)_Q,
    \end{split}
\end{equation*}
then $\lVert\psi-\psi^h\rVert_{L^\infty((0,T),L^2)}\leq Ch.$

\end{lemm}

\begin{proof}
Let $\xi = \psi - \psi^h$. Then due to 
\begin{equation*}
    \begin{split}
        \partial_t\xi &= (f*|\phi_0|^2+V-f_Q*|\phi_0|_Q^2-V_Q)\psi+(f_Q*|\phi_0|_Q^2+V_Q)\xi
        \end{split}
\end{equation*}
we find that 
\begin{equation*}
    \begin{split}
        \frac{d}{dt}\lVert\xi\rVert_{2} &\leq \lVert (f*|\phi_0|^2+V-f_Q*|\phi_0^h|^2-V_Q)\psi\rVert_{2}+\lVert (f_Q*(|\phi_0|^2)_Q+V_Q)\xi\rVert_{2}\\
        &\leq \lVert f*|\phi_0|^2+V-f_Q*(|\phi_0|^2)_Q-V_Q\rVert_{\infty}\lVert \psi\rVert_{2} + \lVert f_Q*(|\phi_0|^2)_Q+V_Q \rVert_{\infty} \lVert \xi \rVert_{2}    \end{split}
\end{equation*}
and hence by Gronwall's inequality
\begin{equation*}
    \begin{split}
        \sup_{t \in (0,T)}\lVert\xi(t)\rVert_{2} &\leq \left( \lVert \phi_0 - \phi_0^h \rVert_{2} + \lVert f*|\phi_0|^2+V-f_Q*(|\phi_0|^2)_Q-V_Q\rVert_{\infty}\lVert\psi\rVert_{L^\infty((0,T),L^2)}T   \right)\\&\quad\times \exp\left(\lVert f_Q*(|\phi_0|^2)_Q-V_Q\rVert_{\infty}T\right).
    \end{split}
\end{equation*}

Now we need to use \cite[Proposition 8.9]{BH} and the mean value theorem to have, for constants $C$
\[\lVert \phi_0 -\phi_0^h\rVert_{2} \leq C\lVert \phi_0 \rVert_{H^1}h\text{ and }
        \lVert V - V_Q \rVert_{\infty} \leq C\lVert V\rVert_{W^{1,\infty}}h.\]
Finally for the convolution $ \lVert f*|\phi_0|^2-f_Q*(|\phi_0|^2)_Q \rVert_{\infty} \leq C\lVert f \rVert_{\infty}\lVert \phi_0\rVert^2_{2}h.$ Hence, there is $C=C(\lVert\phi_0\rVert_{H^{1}},\lVert V\rVert_{W^{1,\infty}},\lVert f\rVert_{W^{1,\infty}},T)$ such that
\[ \sup_{t \in (0,T)}\lVert \xi(t)\rVert_{2} \leq Ch.\]
\end{proof}
Now we want to put all this together to get a scheme for calculating an approximation to the solution. 

As an intermediate step, we use a continuous (in space) Strang splitting scheme
\begin{equation}
    \begin{split}
    \label{eq:contStrang}
        \psi^-_{k+\frac{1}{2}} &= e^{-\frac{i\tau (-i\nabla-A)^2}{2}}\psi_S(t_k)\\
        \psi^+_{k+\frac{1}{2}}&=\exp\left(-i\tau \mathcal V(\psi^-_{k+\frac{1}{2}})\right)\psi^-_{k+\frac{1}{2}}\\
        \psi_S(t_{k+1}) &= e^{-\frac{i\tau (-i\nabla-A)^2}{2}}\psi^+_{k+\frac{1}{2}}
    \end{split}
\end{equation}
where $\psi_S(0)=\phi_0$ and $\mathcal V(\psi):=f*|\psi|^2+V$. Finally, we want to approximate this scheme using the cubic discretization and a Crank-Nicholson method for the propagation of the linear kinetic operator $(-i\nabla-A)^2.$ In particular,
\[\mathcal C^{\tau}_h \phi:=\left(1-\frac{i\tau}{2}H^h_{\operatorname{lin}} \right)\left(1+\frac{i\tau}{2} H^h_{\operatorname{lin}}\right)^{-1} \phi.\]
Recall that for $\lambda \in \mathbb R$ and $t \in [0,T]$ there is $C_T>0$ such that
\[ \left\vert e^{-it\lambda}-\left(\frac{1-i\frac{t}{2} \lambda}{1+i\frac{t}{2} \lambda} \right) \right\vert  \le C_T t^3 \vert \lambda \vert^3,\]
we see, using functional calculus that there is (Taylor expansion) some constant $C_T>0$ independent of $h$ and $\tau \in [0,T]$ such that 
\begin{equation}
\begin{split}
\label{eq:CNscheme}
 \left\lVert \left(e^{-iH_{\operatorname{lin}}^h\tau}  - \mathcal C^{\tau}_{h} \right)\psi^h (t_{k}) \right\rVert_{L^2(\Omega)}\le C_T  \tau^3/h^6.
\end{split}
\end{equation}
This yields a Strang splitting scheme
\begin{equation}
    \begin{split}
    \label{eq:Strang}
        \phi^-_{k+\frac{1}{2}} &= \mathcal{C}_h^{\tau/2}\phi_S(t_k)\\
        \phi^+_{k+\frac{1}{2}}&=\exp_K\left(-i\tau \mathcal V_Q(\phi^-_{k+\frac{1}{2}})\right)\phi^-_{k+\frac{1}{2}}\\
        \phi_S(t_{k+1}) &= \mathcal{C}_h^{\tau/2}\phi^+_{k+\frac{1}{2}}
    \end{split}
\end{equation}
 
 where $\psi_S(0)=(\phi_0)_Q$ and $\mathcal V_Q(\psi_Q):=f_Q*(|\psi|^2)_Q+V_Q$.

\begin{theo}
\label{theo:numerics}
Given an initial state $\phi_0\in H^2(\Omega) \cap H^1_0(\Omega)$ and  $V,f \in W^{1,\infty}(\Omega)$. There is a constant  $C=C(T,\tau,\lVert\phi_0\rVert_{H^2},\lVert V\rVert_{W^{1,\infty}},\lVert f\rVert_{W^{1,\infty}})$ such that the solution $\phi_S$ obtained from the Strang splitting scheme     \eqref{eq:Strang} and the full solution $\psi$ satisfy
\begin{equation}
\label{eq:Lipschitz}
\max_{k ; 0 \le k \le \lfloor T/\tau \rfloor} \rVert \psi(\tau k)-\psi_S(\tau k)\rVert_{2} \leq C\lVert \phi_0-(\phi_0)_Q\rVert_{2}.
\end{equation}
\end{theo}

\begin{proof}
We have already shown the convergence to the continuous Strang splitting scheme \eqref{eq:contStrang} in Lemmas \ref{lemm:lemm1} and \ref{lemm:lemm2}. It therefore suffices to approximate the continuous splitting scheme by the discrete one
. 
The $L^2$-convergence to the discrete approximation by the Crank-Nicholson method in the discrete Strang splitting scheme \eqref{eq:Strang}, follows already from \eqref{eq:CNscheme}. Thus, it only remains to prove the convergence of the propagation of the non-linearity:
\begin{equation*}
    \begin{split}
        \lVert &\left(e^{-i\tau \mathcal V(\psi^-_{k+1/2})}-e^{-i\tau \mathcal V_Q(\phi^-_{k+1/2})}\right)\phi^-_{k+1/2}\rVert_{2} \\&\leq \left\lVert \psi^-_{k+1/2}-\phi^-_{k+1/2}\right\rVert_{2} + \left\lVert \left(e^{-i\tau \mathcal V(\psi^-_{k+1/2})}-e^{-i\tau \mathcal V_Q(\phi^-_{k+1/2})}\right)\phi^-_{k+1/2} \right\rVert_{2}\\
        &\leq \left\lVert \psi^-_{k+1/2}-\phi^-_{k+1/2}\right\rVert_{2}+\tau \Vert f* \vert \psi^-_{k+1/2} \vert^2 -f_Q*\vert \phi^-_{k+1/2}\vert^2 \Vert_{\infty}\\
        &\leq \left\lVert \psi^-_{k+1/2}-\phi^-_{k+1/2}\right\rVert_{2}+\tau \Vert f* (\vert \psi^-_{k+1/2} \vert^2 -\vert \phi^-_{k+1/2}\vert^2) \Vert_{\infty} +\tau \left\lVert f-f_Q \right\rVert_{\infty}\\
        &\leq (1+2\tau \Vert f \Vert_{\infty})\left\lVert \psi^-_{k+1/2}-\phi^-_{k+1/2}\right\rVert_{2} +\tau \left\lVert f-f_Q \right\rVert_{\infty}.
    \end{split}
\end{equation*}

So we have enough for global Lipschitz estimate \eqref{eq:Lipschitz} of the discrete Strang splitting scheme.

\end{proof}

\section{Optimal control theory}
\label{sec:OCT}
In this section we study an abstract optimal control problem (OCP) for the linear Schr\"odinger equation with magnetic field \eqref{eq:bilSchr}. We consider an energy functional
\begin{equation}
\label{eq:controlfunc}
 \mathcal I(u):=\left\langle \psi(T),S \psi(T) \right\rangle+ \kappa \left\lVert u \right\rVert_{H^1_0(0,T)}^2
 \end{equation}
for a positive operator\footnote{it is easy to see that same arguments holds if $S$ is only assumed to be semi-bounded from below} $S : D(S) \subset  L^2(\mathbb R^3) \rightarrow L^2(\mathbb R^3)$ with form domain $D(\sqrt{S})$ continuously embedded in $H^1_1(\mathbb R^3)$ and parameter $\kappa>0$. Here, $\psi$ is the solution to the linear or nonlinear Schr\"odinger equation with initial value $\varphi_0 \in H^2_2(\mathbb R^3)$ and parameters as specified in Assumption \ref{ass:ass1}, cf. Theorem \ref{theo:H2ExistBoundedMag} and Theorem \ref{theo:magfield}.

Examples of such optimal control problems are numerous \cite{WG} and the motivation to study such a problem in quantum mechanics is often to engineer a control function in order to reach a designated target state $\zeta \in L^2(\mathbb R^3)$. In this case, we can choose $S$ to be the projection $S=\operatorname{id}- \langle \bullet, \zeta\rangle \zeta$. Other common examples of penalization operators $S$ include the minimization/maximization of some physical observable such as kinetic energy $\left\langle \psi(T),-\Delta \psi(T) \right\rangle$. The term $\kappa \left\lVert u \right\rVert_{H^1_0(0,T)}^2$, appearing in the control functional \eqref{eq:controlfunc}, penalizes control functions of high energy. 

\begin{lemm}
The functional $\mathcal I$ has a minimizer in $H^1_0(0,T)$. 
\end{lemm}
\begin{proof}
The Banach-Alaoglu theorem implies the existence of a $H^1_0(0,T)$-weakly convergent subsequence $u_n \overset{}{\rightharpoonup} u$ such that $\lim_{n \rightarrow \infty} \mathcal I(u_n)=\inf_{w \in H^1_0(0,T)} \mathcal I(w).$
Since $u_n$ is weakly convergent, the sequence is bounded. Hence, the sequence $\psi_n$ of solutions to \eqref{eq:bilSchr} with the above controls $u_n$ and joint initial value $\varphi_0 \in H^2_2(\mathbb R^3)$ is bounded in 
\[W:=\left\{\varphi \in L^{\infty}((0,T), H^2_2(\mathbb R^3)), \varphi' \in L^{\infty}((0,T),L^2(\mathbb R^3)) \right\}.\]
As $H^2_2(\mathbb R^3)$ is compact in $H_{2-\varepsilon}^{2-\varepsilon}(\mathbb R^3)$ the Aubin-Lions lemma implies that for yet another subsequence of $(\psi_n)$
\[\lim_{n \rightarrow \infty}\left\lVert  \psi_n-\psi \right\rVert_{C((0,T); H_{2-\varepsilon}^{2-\varepsilon}(\mathbb R^3))}=0\]
where $\psi$ is the solution to the Schr\"odinger equation with initial value $\varphi_0$ and control $u.$ Weak lower semicontinuity of the $H^1_0(0,T)$ norm and $H_{2-\varepsilon}^{2-\varepsilon}$ convergence of $(\psi_n)$ implies that $\mathcal I(u) = \inf_{w \in H^1_0(0,T)} \mathcal I(w).$
\end{proof}

The \emph{adjoint equation} to the linear magnetic Schr\"odinger equation is defined as
\begin{equation}
\begin{split}
\label{eq:adjoint}
i \partial_t \zeta &=(-i\nabla-A)^2\zeta + V\zeta + V_{\operatorname{TD}}(t)\zeta, \ (x,t) \in \mathbb R^3 \times (0,T) \\
\zeta(T)&=S\psi(T).
\end{split}
\end{equation}
It is easy to see that the vanishing of the Gateaux derivative of the cost functional $(D_h\mathcal I)(u)=0$ for all admissible functions $h$ is equivalent to a differential equation which is often referred to as the \emph{optimality equation}. For the linear Schr\"odinger equation this equation reads  
\[-2\kappa u''(t) =\left\langle \zeta(t), V_{\operatorname{con}} \psi(t) \right\rangle_{L^2(\mathbb R^3)}, \ u \in H^1_0(0,T),\] 
with parameter $\kappa$ as in the control functional \eqref{eq:controlfunc}, and $\zeta$ the solution to the adjoint problem.
Various algorithms \cite{IK1,IK2,WG} to numerically solve the OCP \eqref{eq:controlfunc} rely on solving the triplet consisting of Schr\"odinger equation \eqref{eq:bilSchr}, adjoint equation \eqref{eq:adjoint}, and optimality differential equation. The following example shows that the optimal control can in general not be computed by that approach.

Our findings of the previous sections showed that it is possible to compute the solution to the Schrödinger equation by restricting it to an auxiliary BVP. The following example shows that this is not possible for optimal control problems.
\begin{ex}
Consider an odd control potential $V$ and an even initial state $\varphi_0$ to the Schr\"odinger equation
\begin{equation*}
\begin{split}
i \psi_u'(t,x) = \left(-\partial_x^2 + u(t)V\right) \psi_u(t,x), \quad \psi(0) =\varphi_0.
\end{split}
\end{equation*}
We also take $S$ to be an even function in the window that the algorithm samples such that there exists a non-zero optimal control with $L^2$ distance to any other optimal control of at least $2\varepsilon.$ Our assumptions guarantee that if $u$ is an optimal control, then so is $-u$ as the solution satisfies $\psi_u(x)= \psi_u(-x)$ and $S$ is assumed to be an even function in the window the algorithm samples.

It suffices now to continue $S$ in such a way that the control functional \eqref{eq:controlfunc} favors $u$ over $-u$. The algorithm necessarily has to include both $u$ and $-u$ and will therefore be off by a Hausdorff distance $2\varepsilon$ at least.
\end{ex}
Thus, it is impossible to numerically approximate optimal controls numerically by sampling only a finite amount of data. However, it is possible to approximate the value of the control functional for a positive Schr\"odinger operator $S=(-i\nabla-A)^2+ U$ with magnetic potential $A$ and control potential $U$ (possibly different from the potentials governing the evolution of the equation) satisfying the conditions of the static potential in Assumption \ref{ass:ass1}.

\begin{proof}
The result follows straight from the result we have already established in previous sections:
We first observe that by Theorems \ref{theo:magfield} and \ref{theo:numerics}, we can compute up to any precision the value of the control functional \eqref{eq:controlfunc} for any fixed control $u$ in the domain of the computational problem.
Thus, by computing the value of the control functional for any fixed control, we obtain an upper bound on the minimal value of \eqref{eq:controlfunc} and thus also on the $H^1_0$ bound of any optimal control. This yields a bound on the absolute value of Fourier coefficients on the optimal control. This implies an error that does not exceed $\varepsilon$ in the optimal control functional, as the solution is easily seen to be Lipschitz continuous with respect to controls in $L^1(0,T)$.\end{proof}

\section{Numerical examples}
\label{sec:numex}
\begin{figure}[ht!]
   \subfloat[\label{genworkflow}]{%
      \includegraphics[clip, width=0.3\textwidth]{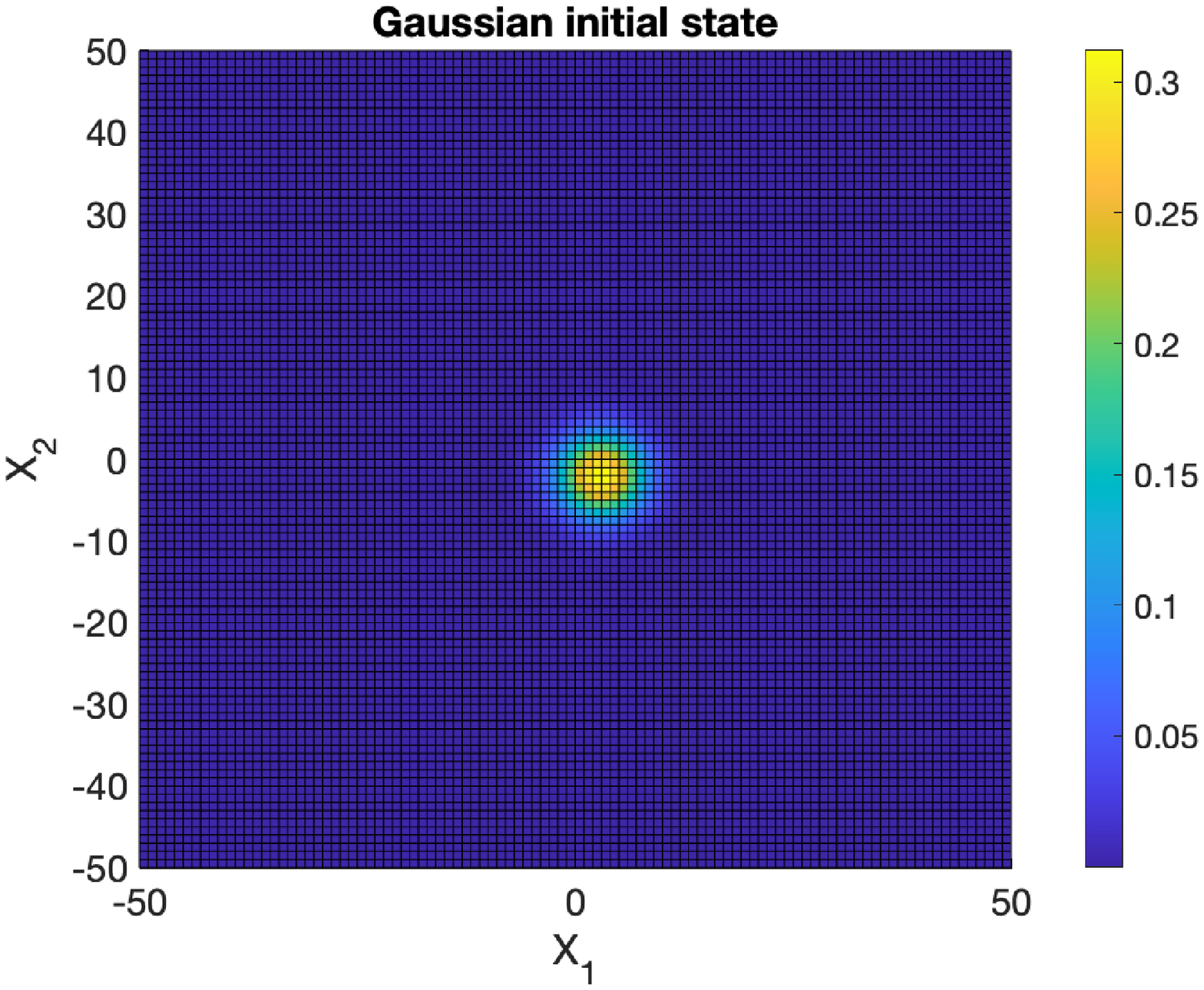}}
\hspace{\fill}
   \subfloat[\label{pyramidprocess} ]{%
      \includegraphics[clip, width=0.32\textwidth]{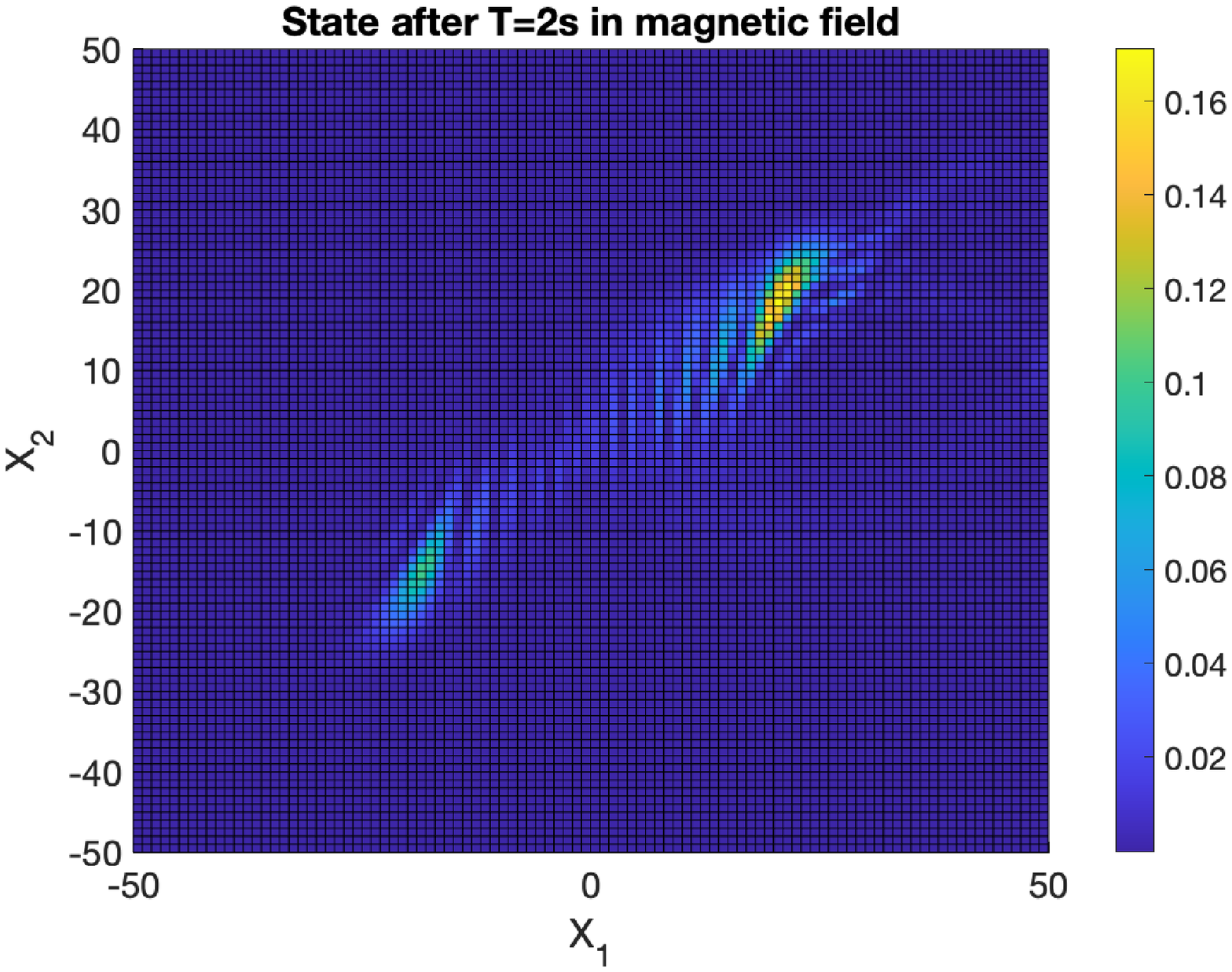}}
\hspace{\fill}
   \subfloat[\label{mt-simtask}]{%
      \includegraphics[clip, width=0.32\textwidth]{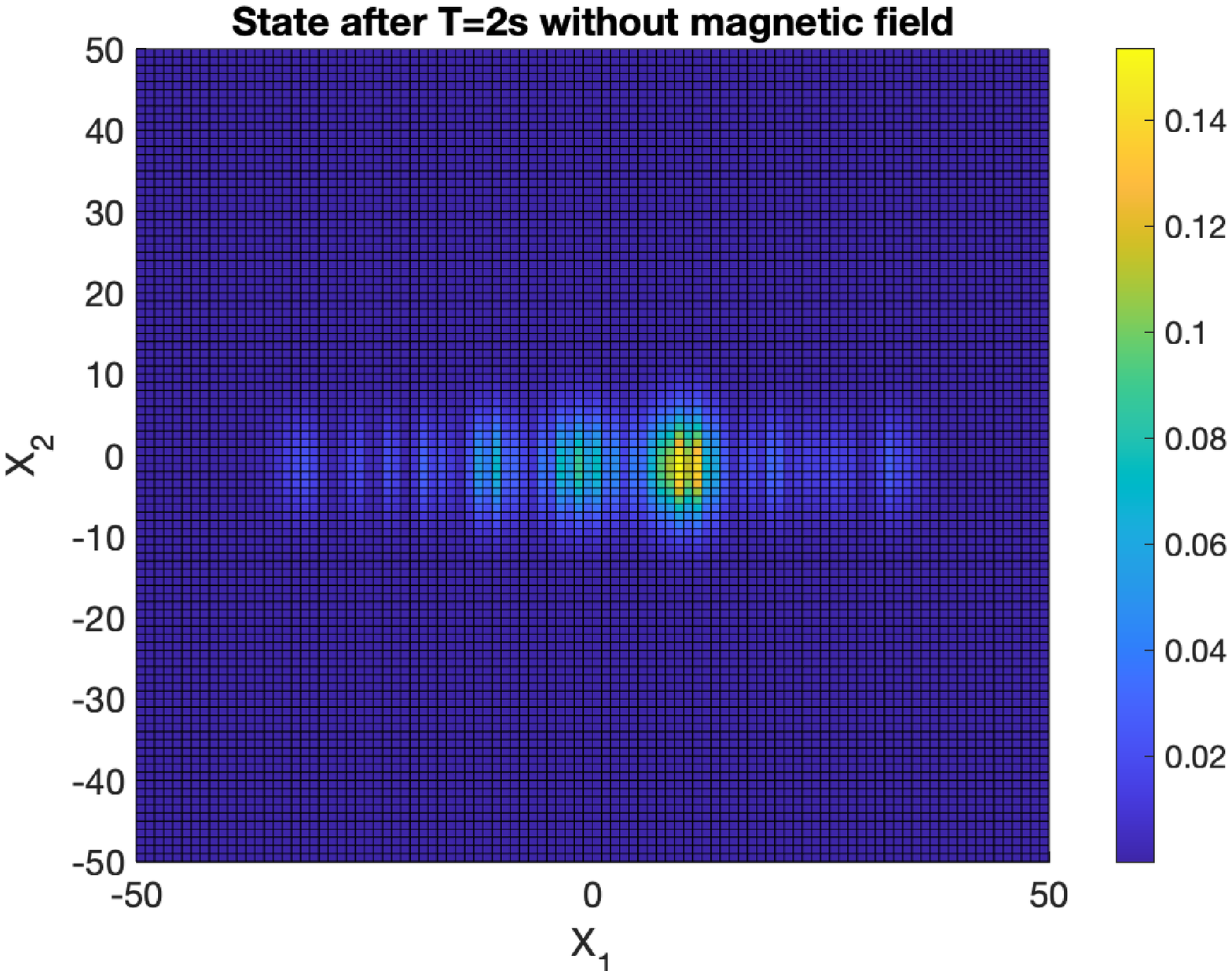}}\\
\caption{\label{workflow}Propagation of Gaussian initial state \eqref{eq:Gaussian} under the linear Schrödinger equation. The colormap depicts the probability amplitude of the state in the respective square (a) initial state at time $t=0$; (b) evolution after time $t=2s$ with magnetic field where the Lorentz force shifts the direction of propagation; (c) evolution after time $t=2s$ without magnetic field.}
\label{fig:propagation}
\end{figure}
We consider a Gaussian initial state 
\begin{equation}
    \label{eq:Gaussian}
\psi_0(x) = \frac{1}{\pi} \operatorname{exp}\left(-\frac{(x_1+5/2)^2}{2}-\frac{(x_2-5/2)^2}{2} \right)\end{equation}
and a potential that is quadratically confining in $x_1$ direction and quadratically spreading (exponentially fast in the dynamics) in $x_2$ direction
$$V_{\operatorname{reg}}(x) = x_1^2-x_2^2.$$ In addition, we consider a constant magnetic field of unit strength $B_0=1.$
We then simulate the dynamics for two seconds with $100$ time steps with Crank-Nicholson method and Strang splitting scheme in Figure \ref{fig:propagation}. The wavepacket then propagates exponentially fast away from the origin, which is consistent with the Gronwall estimates from our theoretical part, but the entire dynamics can still be captured in a bounded domain with Dirichlet boundary conditions.


\begin{thebibliography}{0}
\bibitem[B05]{B} Baudouin, L. (2005). \emph{Existence and Regularity of the Solution of a Time Dependent Hartree-Fock Equation Coupled with a Classical Nuclear Dynamics}, Rev. Mat. Complut. 2005, 18; Num. 2, 285-314.

\bibitem[BD11]{BD}Bao, W. and Dong, X. (2011). \emph{Numerical methods for computing ground states and dynamics of nonlinear relativistic Hartree equation for boson stars}. Journal of Computational Physics. Elsevier.


\bibitem[BH20]{BH} Becker, S. and Hansen, A. (2020). \emph{Computing solutions of Schr\"odinger equations \\- On the brink of numerical algorithms}, \arXiv{2010.16347}.

\bibitem[BKP05]{BKP} Baudouin, L., Kavian, O., and Puel, J. (2005). \emph{Regularity for a Schr\"odinger equation with singular potentials and application to bilinear optimal control}, Journal of Differential Equations, Volume 216, Issue 1, Pages 188-222.


\bibitem[BM91]{BM} Brezzi, F. and Markowich, P. \emph{The three-dimensional Wigner-Poisson problem: Existence,
uniqueness and approximation}, Math. Methods Appl. Sci. 14 (1991), 35-61. 

\bibitem[DM94]{DM} Hans De Raedt, and Kristel Michielsen. (1994) \emph{Algorithm to solve the time-dependent Schrödinger equation for a charged particle
in an inhomogeneous magnetic field: Application to the Aharonov–Bohm effect}.
Computers in Physics 8, 600 

\bibitem[F04]{F} Fröhlich, M. (2004). \emph{Exponentielle Integrationsverfahren für die Schrödinger-Poisson-Gleichung},
Doctoral Thesis, Universität Tübingen.

\bibitem[IK07]{IK1} Ito, K. and Kunisch, K. (2007). \emph{Optimal Bilinear Control of an Abstract Schr\"odinger Equation.}  SIAM Journal on Control and Optimization, Vol. 46, No. 1 : pp. 274-287.

\bibitem[IK09]{IK2} Ito, K. and Kunisch, K. (2009). \emph{Asymptotic properties of feedback solutions for a class of quantum control problems.}

\bibitem[IS63]{IS} Segal, I. (1963) \emph{Non-Linear Semi-Groups}. Annals of Mathematics, Second Series, Vol. 78, No. 2, pp. 339-364.

\bibitem[IZL94]{IZL} R. Illner, P.F. Zweifel, H. Lange, \emph{Global existence, uniqueness and asymptotic behaviour of solutions of the Wigner-Poisson and Schrödinger-Poisson systems}, Math. Methods Appl. Sci. 17 (1994), 349–376.

\bibitem[JG07]{JG07}Jiang, S. and Greengard, L. (2007). \emph{Efficient representation of nonreflecting boundary conditions for the time‐dependent Schrödinger equation in two dimensions}. Communications on Pure and Applied MathematicsVolume 61, Issue 2.

\bibitem[KG18]{KG18}Kaye, J. and Greengard, L. (2020). \emph{Transparent Boundary Conditions for the Time-Dependent Schr\" odinger Equation with a Vector Potential}. \arXiv{1812.04200}.

\bibitem[L08]{L} Lubich, C. (2008). \emph{On splitting methods for Schr\"odinger-Poisson and cubic nonlinear Schr\"odinger equations}, Mathematics of Computation, Volume 77, Number 264, Pages 2141-2153

\bibitem[RS75]{RS} Reed, M. and Simon, B. (1975) \emph{Methods of Modern Mathematical Physics}.

\bibitem[SS07]{Soffer} Soffer, A., and Stucchio, C. \emph{Open boundaries for the nonlinear Schrödinger equation.} Journal of Computational Physics 225.2 (2007): 1218-1232.
\bibitem[SS09]{Soffer2} Soffer, A., and Stucchio, C. \emph{Multiscale resolution of shortwave‐longwave interaction.} Communications on Pure and Applied Mathematics: A Journal Issued by the Courant Institute of Mathematical Sciences 62.1 (2009): 82-124.

\bibitem[WG07]{WG} Werschnik, J. and Gross, E. (2007). \emph{Quantum optimal control theory}. Journal of Physics B: Atomic, Molecular and Optical Physics, Volume 40, Number 18.

\bibitem[WL07]{WK} Wu, X. and Li, X. \emph{Absorbing boundary conditions for the time-dependent Schrödinger-type equations in $\mathbb R^3.$}
Phys. Rev. E 101, 013304.

\bibitem[ZWW19]{ZWW}Zhai, S., Wang, D., Weng, Z. et al. Error Analysis and Numerical Simulations of Strang Splitting Method for Space Fractional Nonlinear Schrödinger Equation. J Sci Comput 81, 965–989 (2019). 
\end{thebibliography}
\end{document}